\documentclass[11 pt]{amsart}
\usepackage{latexsym,amsmath,amsfonts,graphicx}
\usepackage{graphicx}
\usepackage{subfigure}
\usepackage{amssymb,cases}
\usepackage{slashbox}
\usepackage{amstext}
\usepackage[dvips]{color}
\newtheorem{theorem}{Theorem}[section]
\newtheorem{thm}[theorem]{Theorem}
\newtheorem{pro}{Proposition}[section]

\newtheorem{lemma}[pro]{Lemma}
\newtheorem{remark}[pro]{Remark}
\newtheorem{defi}{Definition}[section]

\numberwithin{equation}{section}

\def\mathscr{\mathcal }

\newcommand{\bomega}{{\boldsymbol {\omega}}}

\newcommand{\bi}{{\mathbf{i}}}
\newcommand{\bx}{{\mathbf{x}}}
\newcommand{\by}{{\mathbf{y}}}
\newcommand{\bu}{{\mathbf{u}}}
\newcommand{\bv}{{\mathbf{v}}}

\begin{document}
\title{Lipschitz equivalence of fractals and finite state automaton}
\author{Hui Rao}
\address{Department of Mathematics and Statistics, Hua Zhong Normal University, Wuhan, China.}
\email{hrao@mail.ccnu.edu.cn}
\author{Yunjie Zhu$\dag$}
\address{Department of Mathematics and Statistics, Hua Zhong Normal University, Wuhan, China.}
\email{yjzhu\_ccnu@sina.com}
\date{\today}
\thanks{$\dag$ The correspondence author.}
\thanks{This work is supported by NSFC Nos. 11431007 and  11471075.}

\begin{abstract}
The study of Lipschitz equivalence of fractals is a very active topic in recent years. Most of the studies in literature
concern totally disconnected fractals.
In this paper, using  finite state automata,
 we construct a bi-Lipschitz map between two fractal squares which are not totally disconnected.
 This is the first non-trivial map of this type. We also show that this map is measure-preserving. 
\end{abstract}


\maketitle
\section{ Introduction}
Two  metric spaces $(X,d_X)$ and $(Y,d_Y)$ are said to be  \emph{Lipschitz equivalent}, and denote by $X\simeq Y$,
if there is a bijection $f:~X\to Y$ and  a constant $C>0$ such that
$$ C^{-1}d_X(x_1,x_2) \leq d_Y \big( f(x_1),f(x_2)\big ) \leq C d_X(x_1,x_2),~~~~\forall~~x_1,x_2\in X.$$
We call $f$ a \emph{bi-Lipschitz mapping} and $C$  a \emph{Lipschitz constant}.
Lipschitz equivalence is an important topic in geometrical measure theory.
The study of Lipschitz equivalence of fractal sets was initialled by by
  Cooper and  Pignartaro \cite{CP}, Falconer and Marsh \cite{FM} and David and Semmes \cite{DS}, \textit{etc};
and it becomes a very active topic recently
(\cite{RRX06, XiXi10,RaoRW12, XiXi12, LuoL13, XiXi13, RuanWX14,FanRZ15, RaoZ15,RaoRW13, LiLM13}).
We note that most of the studies in literature  focus on
self-similar sets  which are totally disconnected, especially a class of fractals called fractal cubes.
For self-similar sets which are not totally disconnected,
 the study is very difficult and  there are few results (\cite{Why58},\cite{RuanW15},\cite{WenZD12},\cite{GuR16}).

An \emph{iterated function system (IFS)} is a family of contractions $\{\varphi_j\}_{j=1}^{m}$  on $\mathbb{R}^{d}$,
and the  \emph{attractor} of  the IFS is the unique nonempty compact set $K$ satisfying
$K=\bigcup_{j=1}^m\varphi_j(K)$, and it is called a \emph{self-similar set}.
See \cite{Fal90} .

For $n\geq 2$ ,~let ${\mathcal{D}}=\{d_1,\cdots,d_m\}  \subseteq \{0,1,\cdots ,n-1\}^d$,
which we call a \emph{digit set}.
Let $\{\varphi_j\}_{j=1}^m$ be the IFS
on $\mathbb{R}^d$ given by $\varphi_j=\frac{1}{n}(x+d_j)$,
then its attractor $K$  satisfies the set equation
\begin{equation}\label{eq-1.1}
 K=\frac{1}{n}(K+{\mathcal{D}}),
\end{equation}
and it is called  a \emph{fractal cube} (\cite{XiXi10}). Especially, when $d=2$, we call $K$ a \emph{fractal square}
(\cite{LuoLR13}).
If two fractal cubes are totally disconnected, there is a simple and elegant criterion for  the Lipschitz equivalence:

\begin{thm}
(  \cite{XiXi10})
Let $E, F$ be two totally disconnected fractal cubes with contraction ratio $1/n$.  Then $E\simeq F$ if
and only if $\dim_H E=\dim_H F$, or alternatively, the digit sets of $E$ and $F$ have the same cardinality.
\end{thm}

\begin{figure}[h]
\centering

\includegraphics[width=3.2cm]{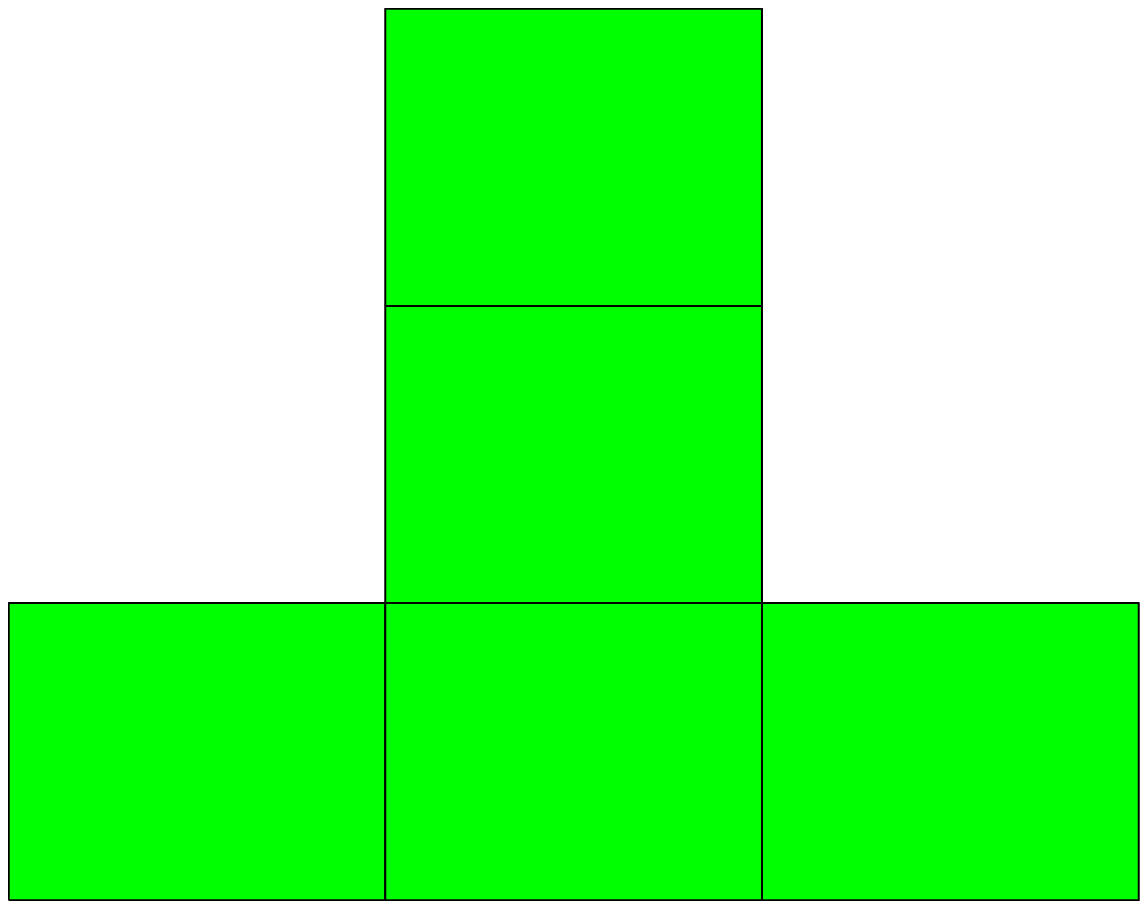} \ \
\includegraphics[width=3.2cm]{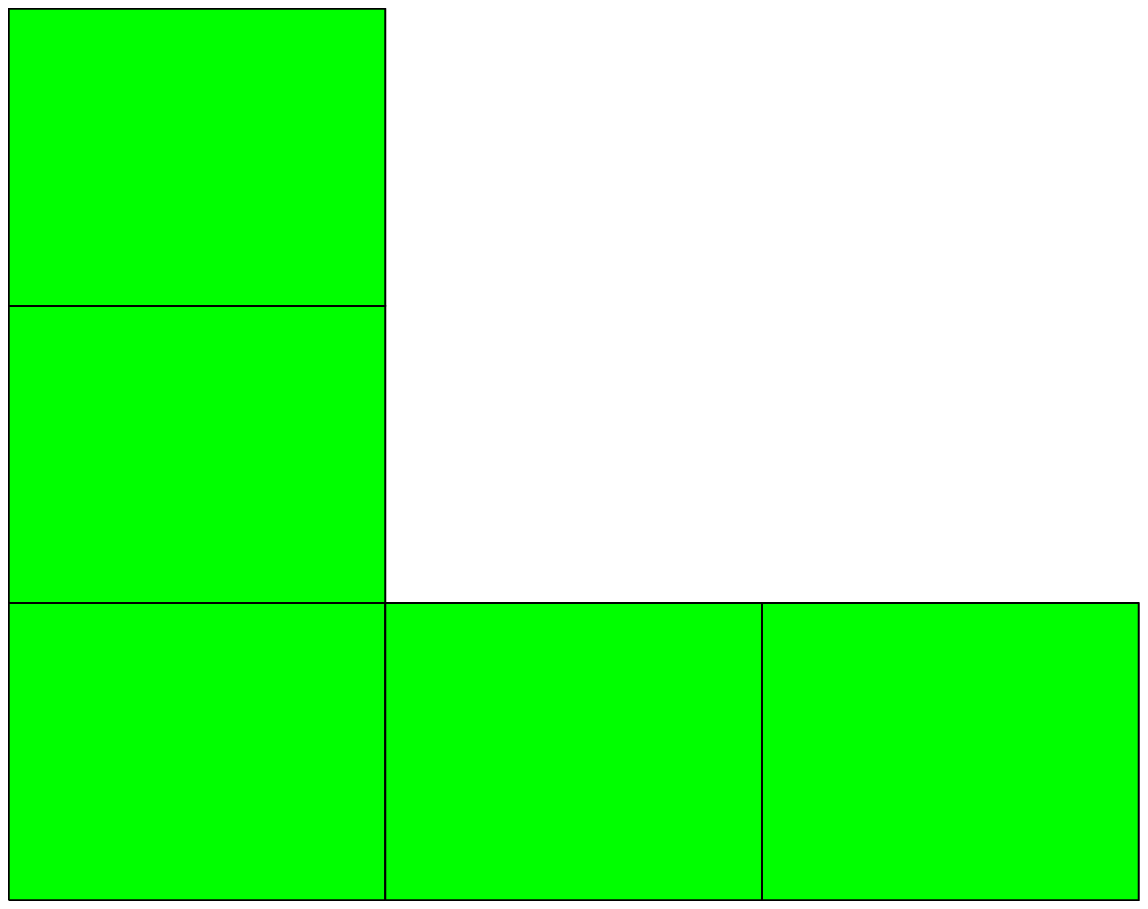} \ \
\includegraphics[width=3.2cm]{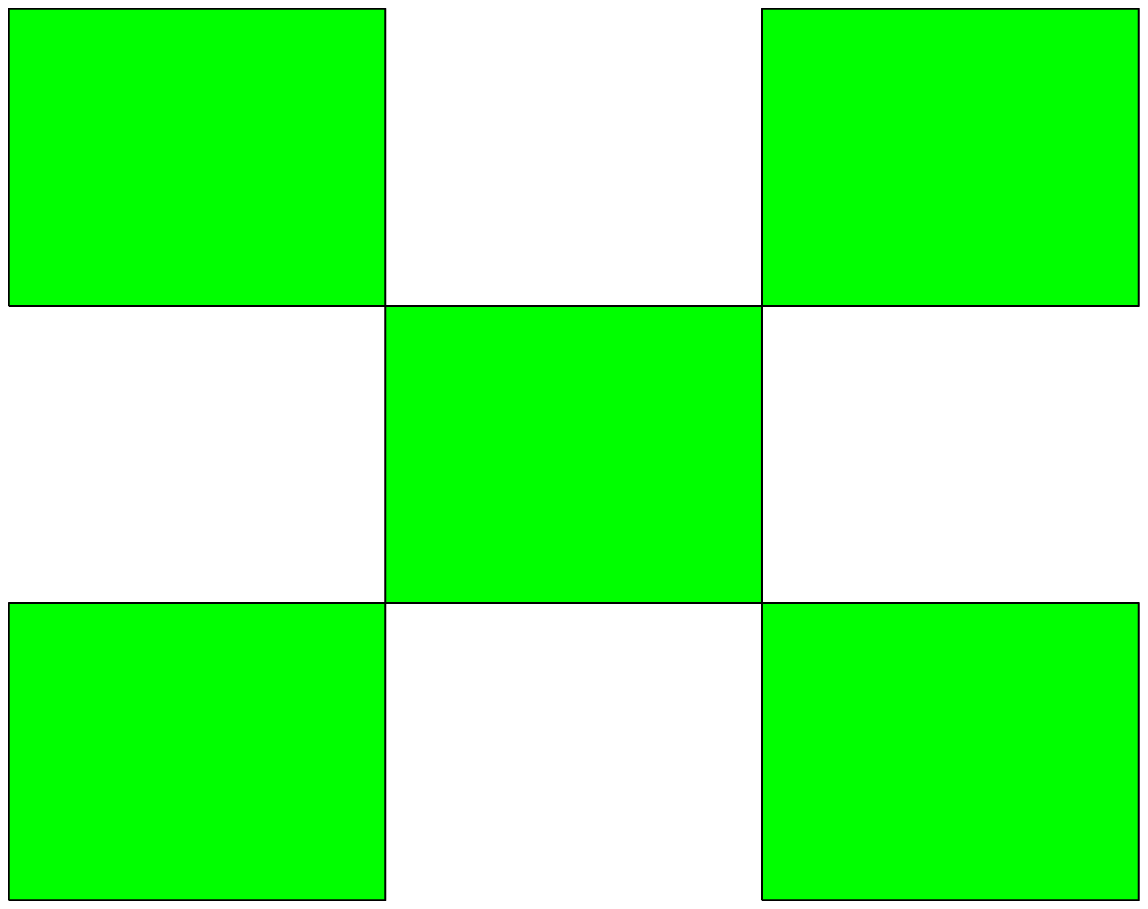}\\
\includegraphics[width=3.2cm]{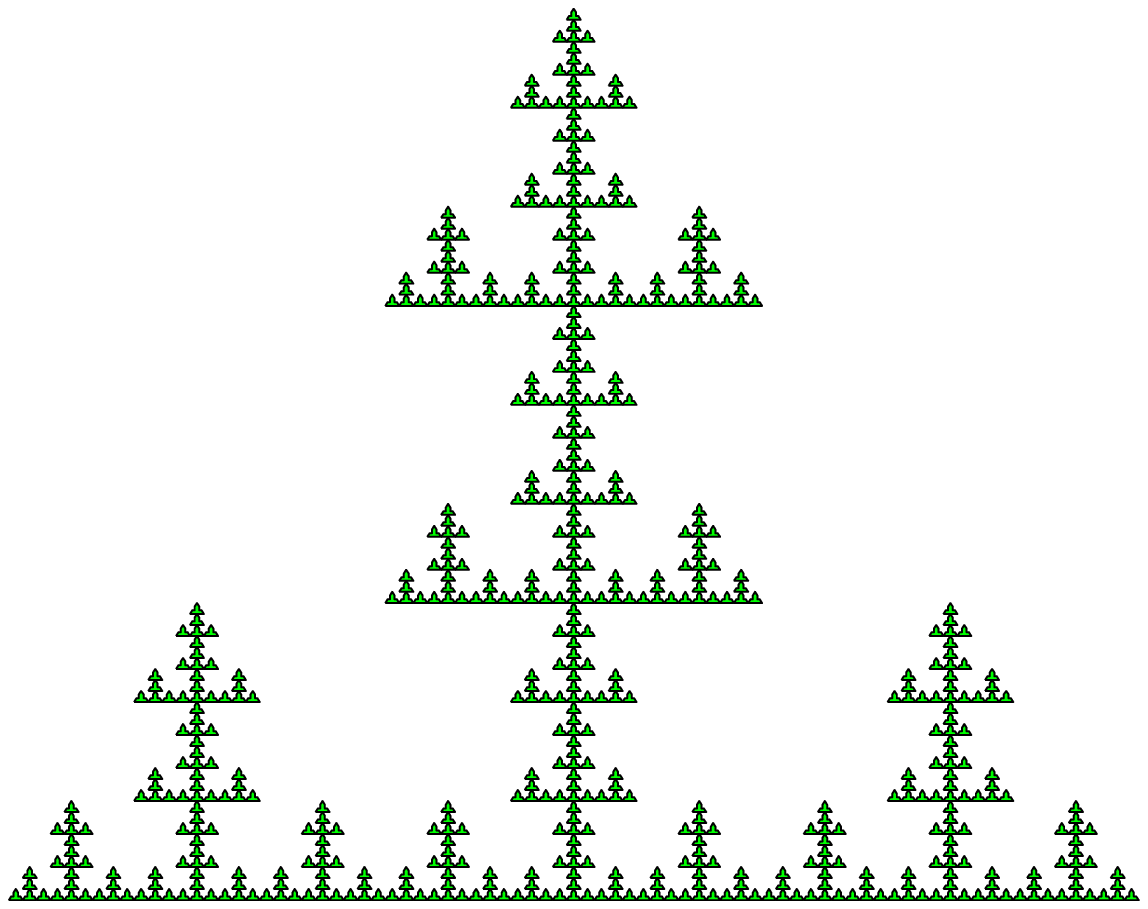} \ \
\includegraphics[width=3.2cm]{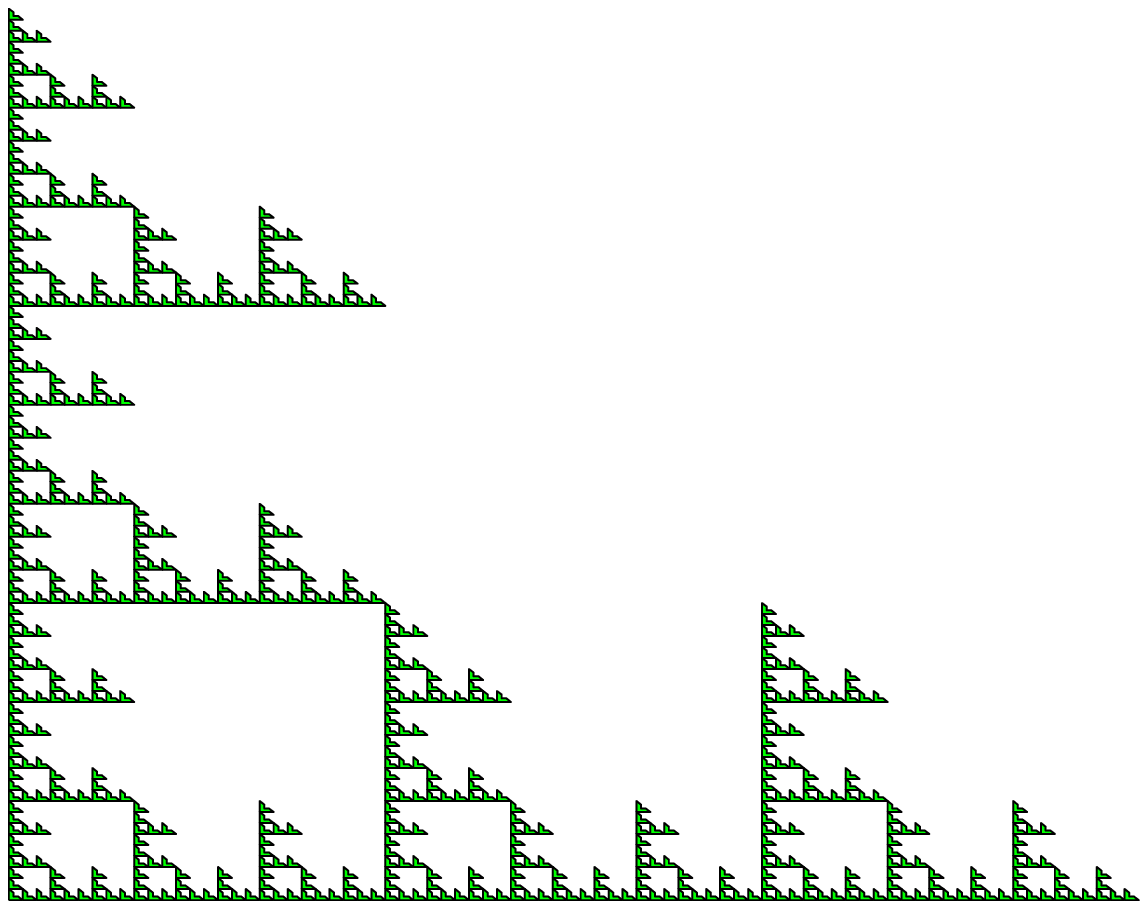} \ \
\includegraphics[width=3.2cm]{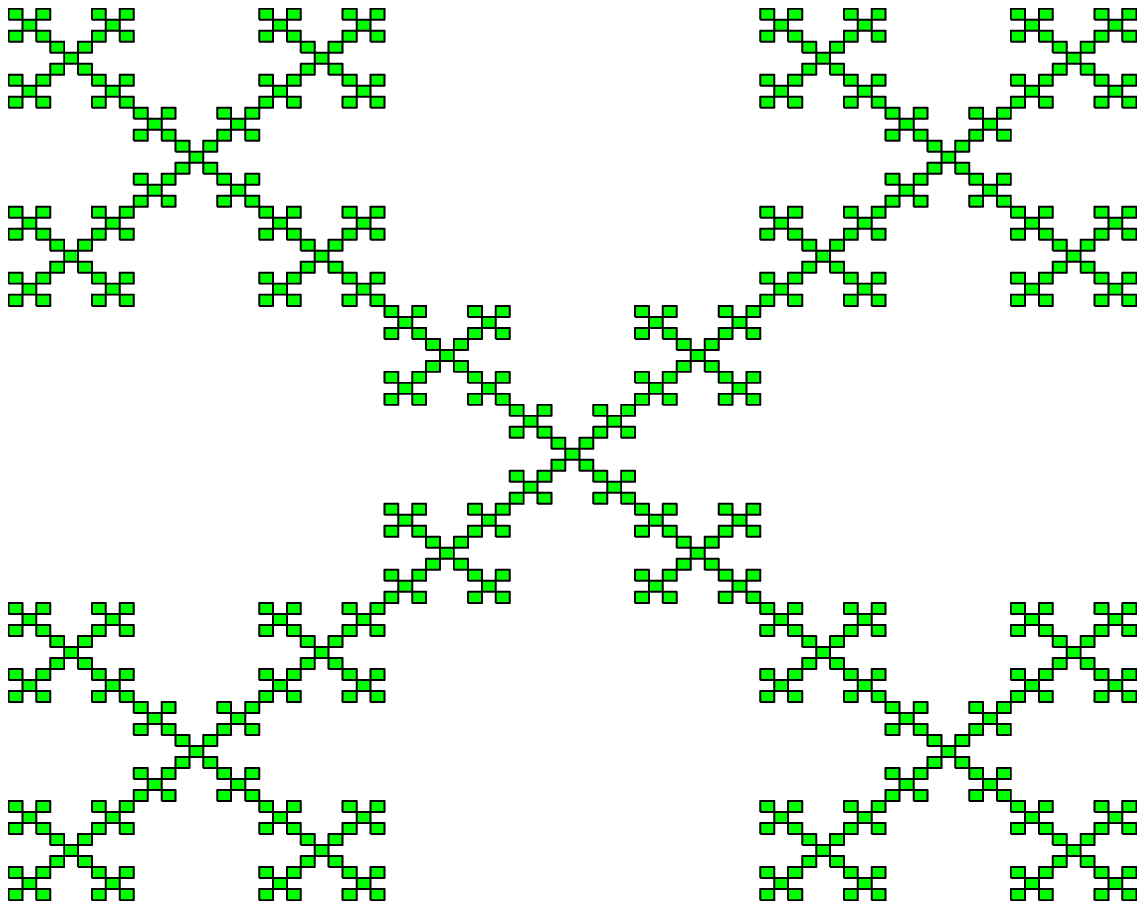}
\label{fig:1}
\caption{Connected fractal squares in ${\mathcal F}_{3,5}$}
\end{figure}

Let $\mathcal{F}_{n,m}$ denote the collection of all fractal cubes satisfying \eqref{eq-1.1} and the cardinality
$\#{\mathcal D}=m$, that is, with contraction ratio $1/n$ and with $m$ branches.
Recently, Luo and Liu \cite{LuoL16} studied the classification of elements in  ${\mathcal F}_{3,5}$ which are not totally disconnected. They show that

$(i)$ for elements in ${\mathcal F}_{3,5}$ which are connected,   there are $4$ classes with respect  to a linear transformation, which are depicted in Figure \ref{fig:1}; moreover,  any two of them are  not homeomorphic.

$(ii)$ For elements in ${\mathcal F}_{3,5}$ which are  not connected but contains non-trivial connected components,
there are $6$ different classes with respect to linear transformations, which are depicted in Figure \ref{fig:2}.

In 2009, Li-feng Xi and Ying Xiong asked the following question in several conferences.

\medskip
\noindent \textbf{Qustion 1.} \textit{Let $F_1$ and $F_2$ be the fractal squares depicted in Figure \ref{fig:2:a} and Figure \ref{fig:2:b}.
 Are they Lipschitz equivalent?}

\medskip

\medskip
\begin{figure}[h]
\centering
\subfigure[ $F_{1}$]{
\label{fig:2:a}
\includegraphics[width=3cm]{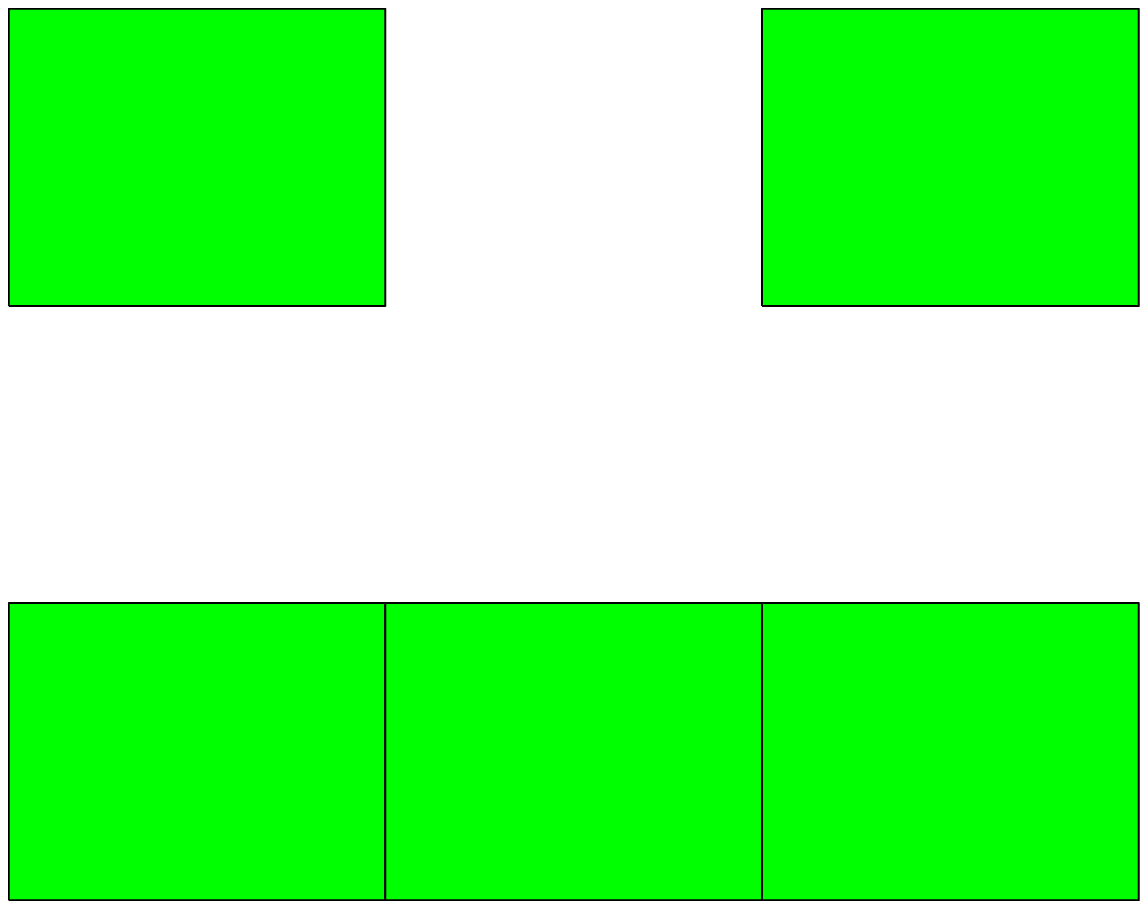}
\includegraphics[width=3cm]{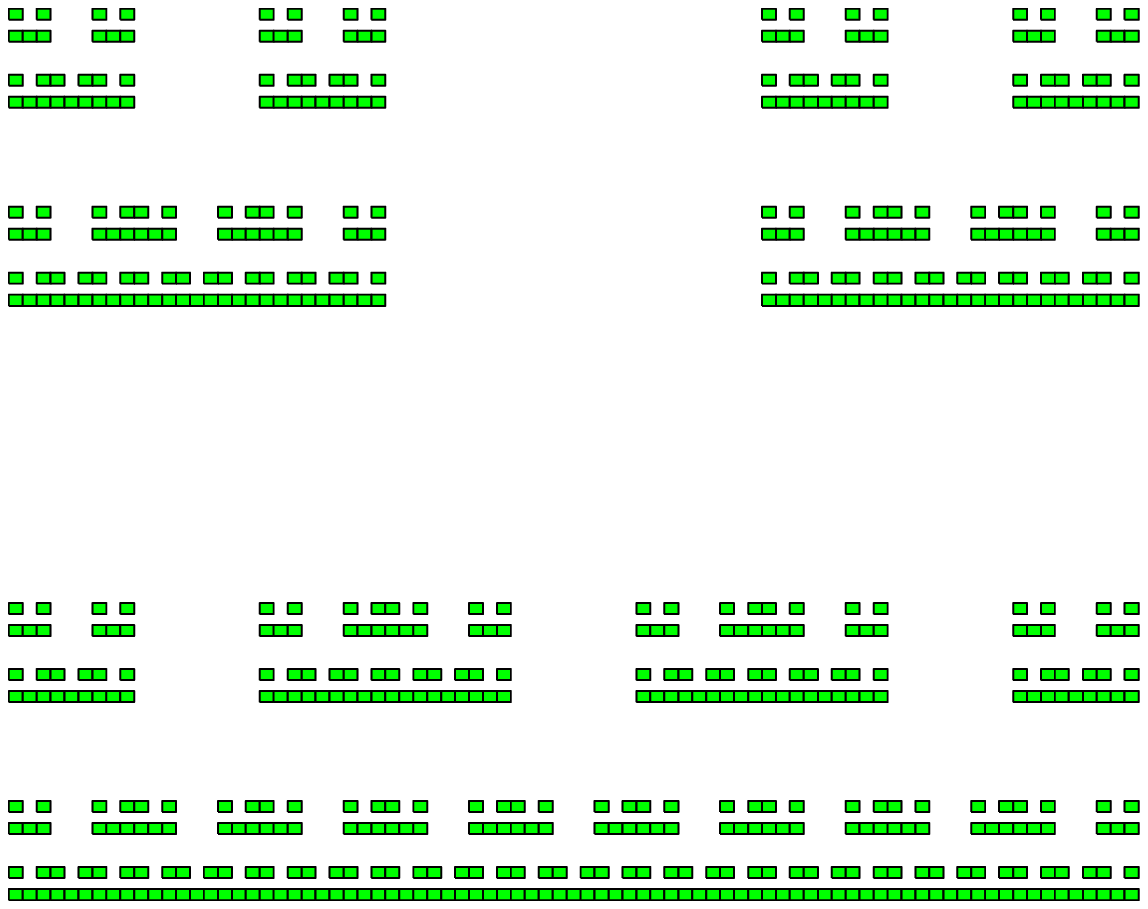}}
\subfigure[ $F_{2}$]{
\label{fig:2:b}
\includegraphics[width=3cm]{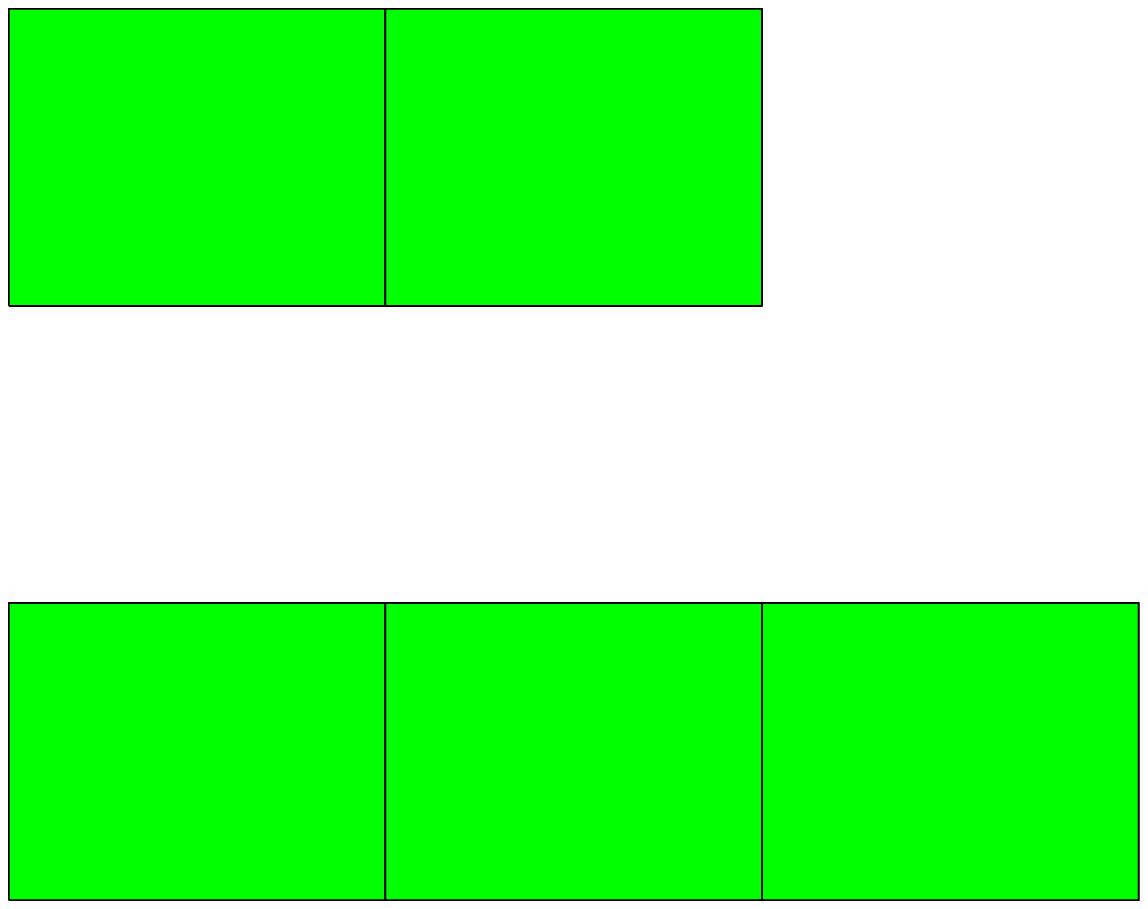}
\includegraphics[width=3cm]{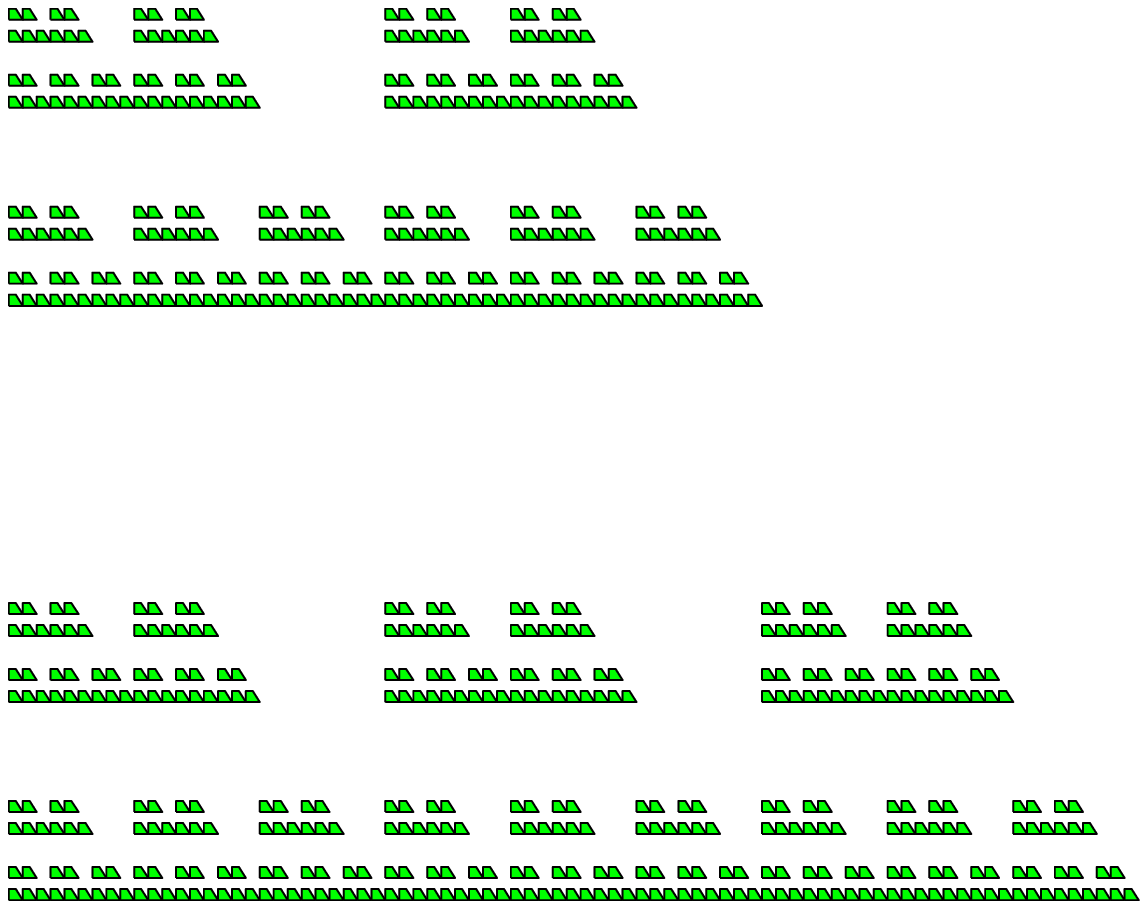}}
\subfigure[ $F_{3}$]{
\label{fig:2:c}
\includegraphics[width=3cm]{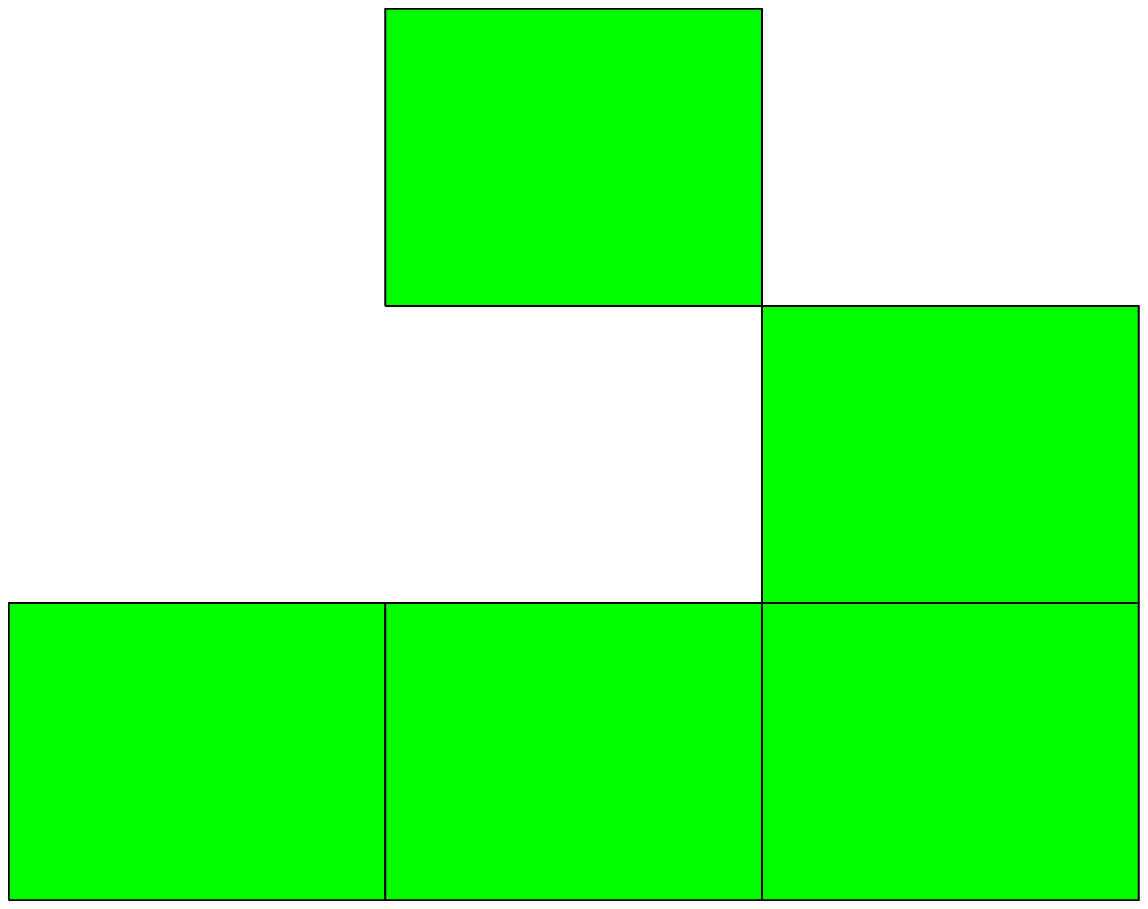}
\includegraphics[width=3cm]{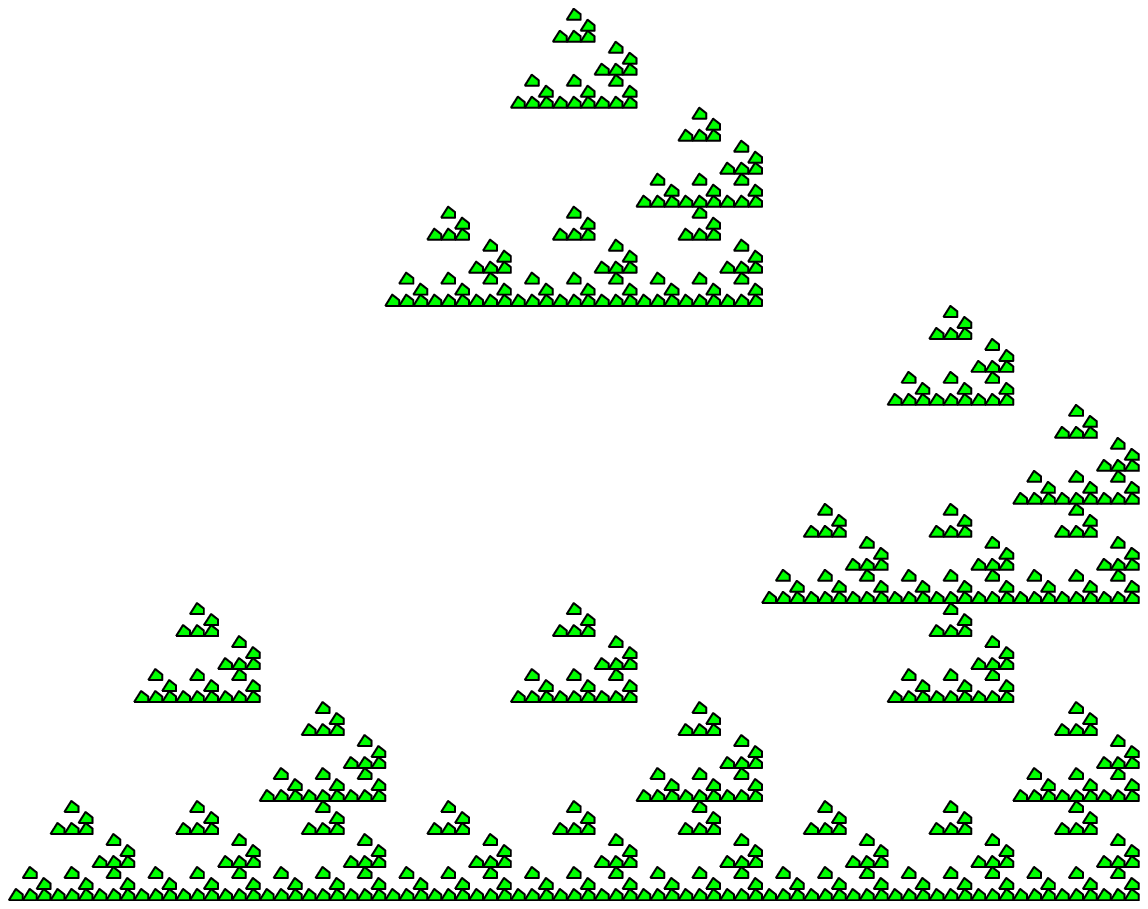}}
\subfigure[ $F_{4}$]{
\label{fig:2:d}
\includegraphics[width=3cm]{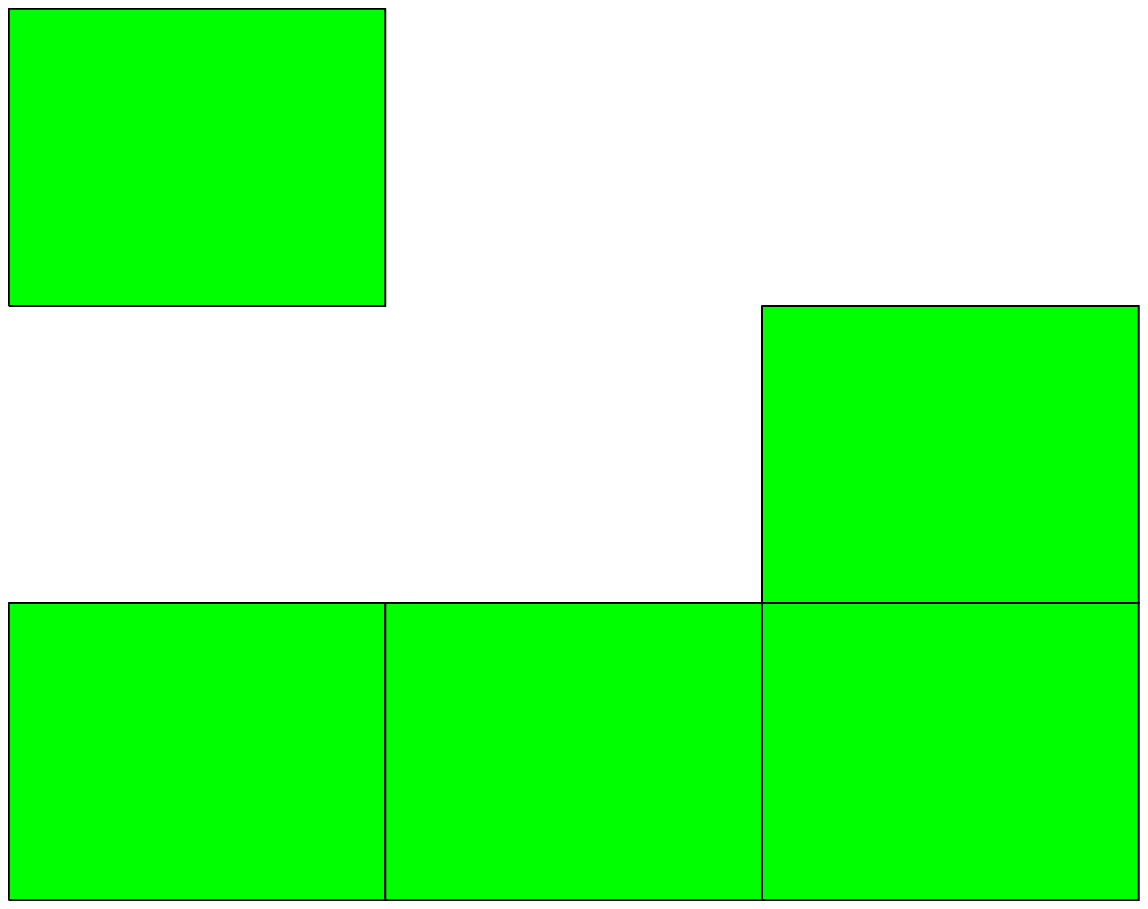}
\includegraphics[width=3cm]{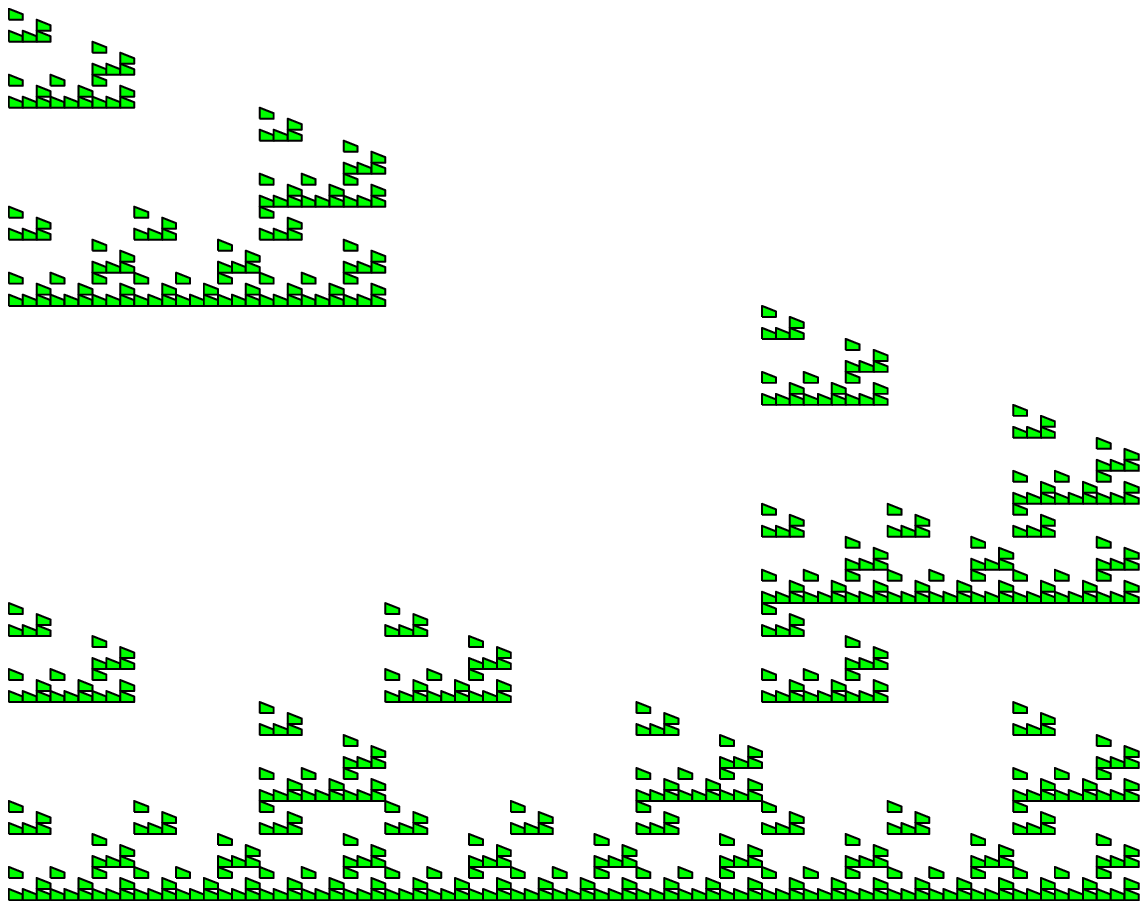}}
\subfigure[ $F_{5}$]{
\label{fig:2:e}
\includegraphics[width=3cm]{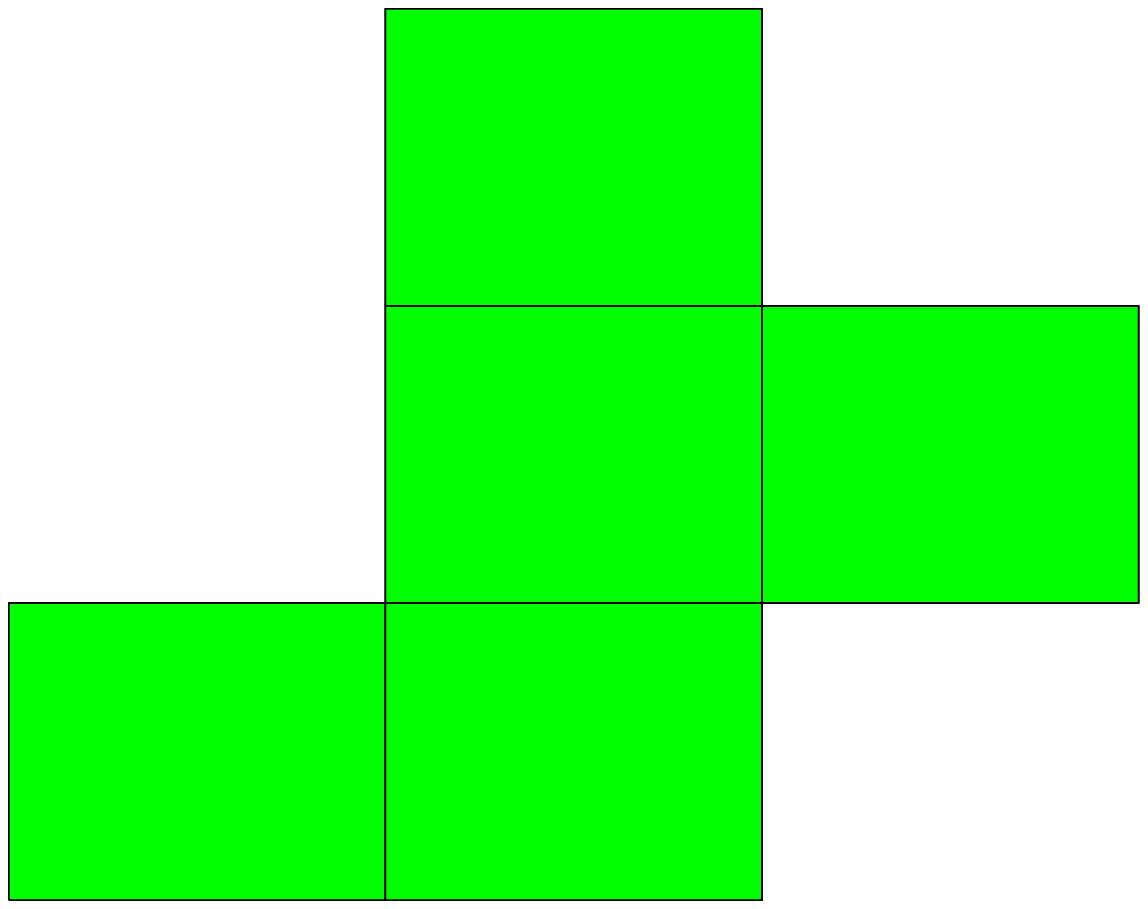}
\includegraphics[width=3cm]{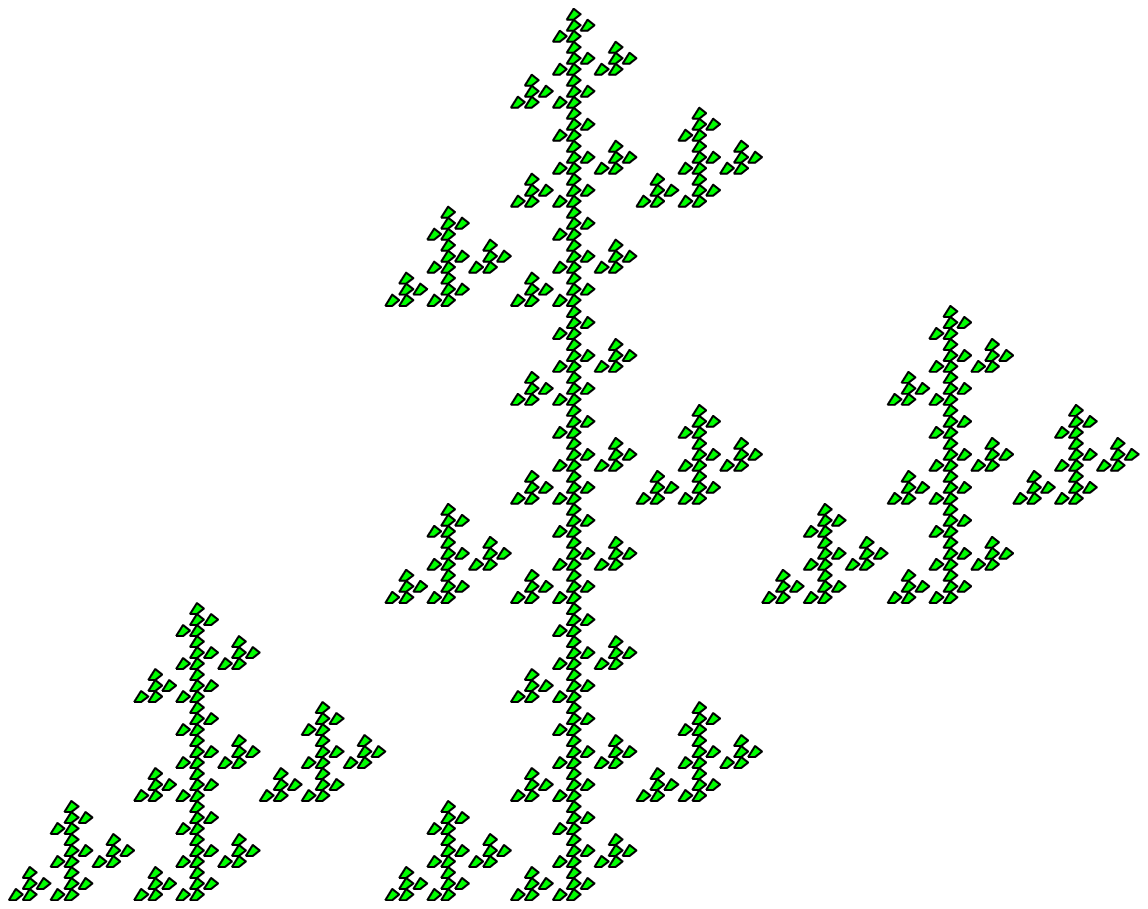}}
\subfigure[ $F_{6}$]{
\label{fig:2:f}
\includegraphics[width=3cm]{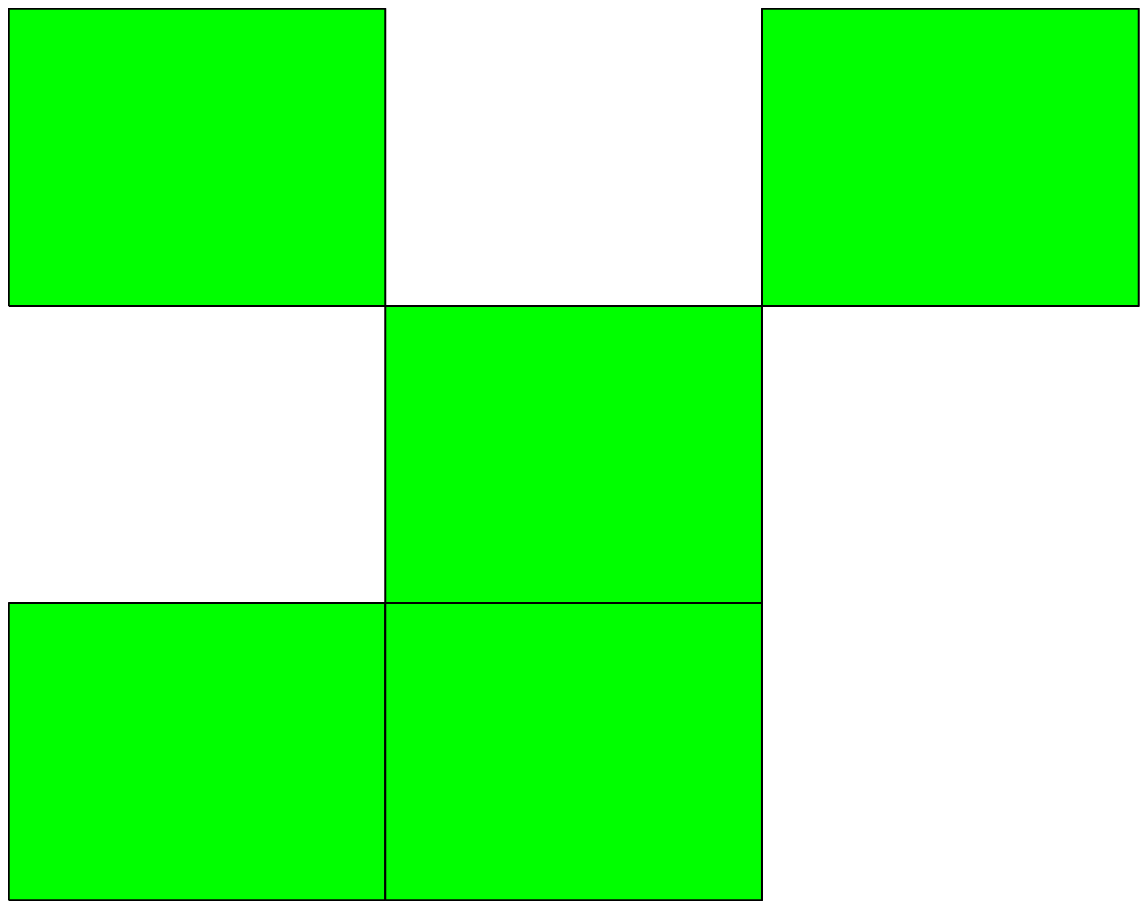}
\includegraphics[width=3cm]{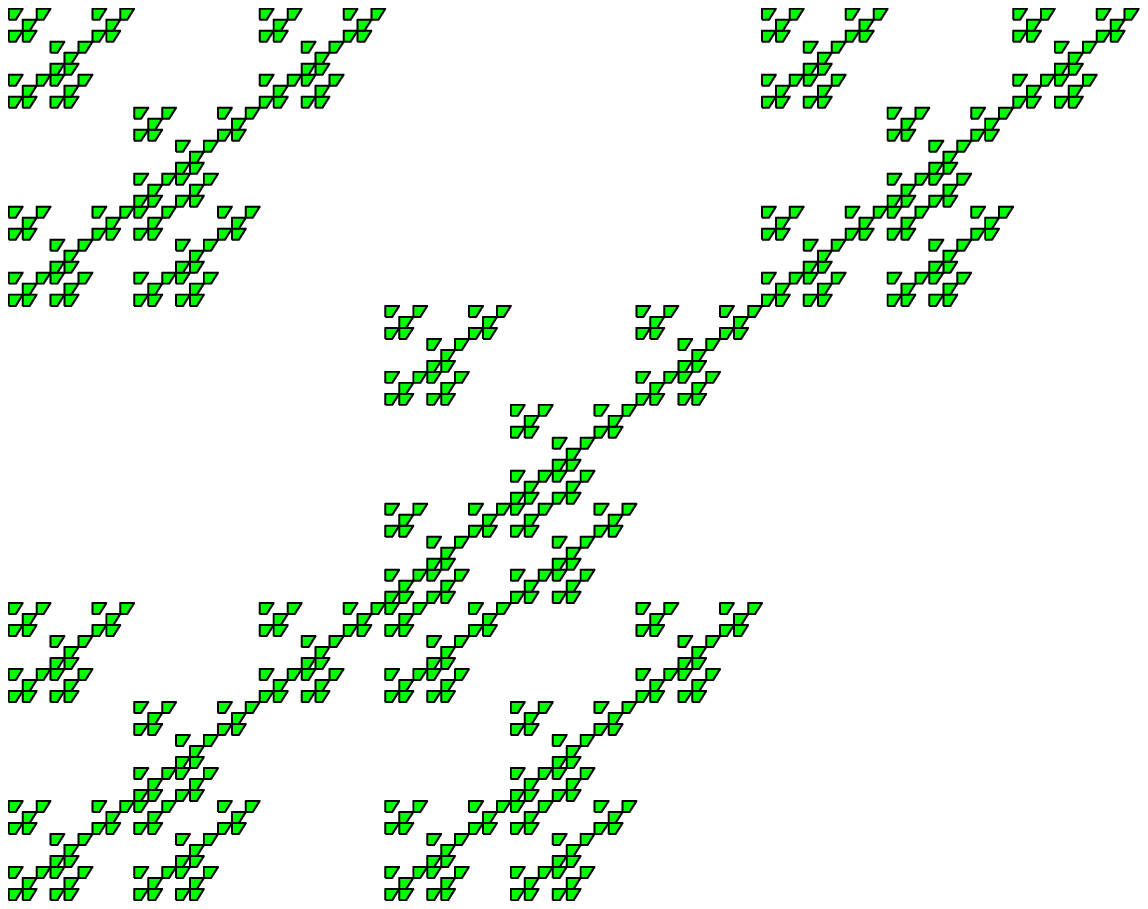}}
\caption{Neither connected nor totally disconnected fractal squares in ${\mathcal F}_{3,5}$.}
\label{fig:2}
\end{figure}

Question 1 is still open. 
The main purpose of this paper is to show that  $F_{3}$ and  $F_{5}$ in Figure \ref{fig:2:c} and Figure \ref{fig:2:e} are Lipschitz equivalent.
For convenience, in what follows,
we shall denote  $F_{3}$  and $F_{5}$ by $E$ and $F$, respectively.
Precisely, the digit set of $E$ is
\begin{equation}\label{1.2}
\mathcal{D}_E=\big\{0,1,2,2+i,1+2i\big\},
\end{equation}
and the digit set of $F$ is
\begin{equation}\label{1.3}
\mathcal{D}_F=\big\{1+2i,1+i,1,0,2+i\big\}.
\end{equation}
Let ${\mathcal H}^s$ denote the $s$-dimensional Hausdorff measure, see \cite{Fal90}.

\begin{theorem}\label{thm-main}
There exists a    bi-Lipschitz map $f:~E\to F$ such that $f$ is measure-preserving in the sense that
for any Borel set  $B\subset E,$
$$
\frac{{\mathcal H}^s(B)}{{\mathcal H}^s(f(B))}=c, 
$$
 where $s=\dim_H E=\log 5/\log 3$ and $c={{\mathcal H}^s(E)}/{{\mathcal H}^s(F)}$.
\end{theorem}

In the study of Lipschitz equivalence of self-similar sets of totally disconnected type,
the strategy is to find  a common graph-directed structures  of the fractals   (see \cite{RRX06}\cite{XiXi10}). However, this method does not apply to fractals which are not totally
disconnected.
In this paper, we introduce some new ideas to deal with the problem.

First, we choose two countable dense subsets $E'\subset E$ and $F'\subset F$.
Secondly, we construct  a Lipschitz mapping between $E'$ and $F'$ via the coding space; to the this end, we
make use of  several finite state automata.  Finally, we extend to mapping to
the whole set by  continuity.

The paper is organized as follows. In Section 2, we construct a  map $f$ between $E'$ and $F'$.
In Section 3, we construct a finite state automaton related to a fractal square, which is a variation of the neighbor graph.
In Section 4, we show that the map $f$ is bi-Lipschitzian and measure-preserving.

\section{\textbf{Constructing a bi-Lipschitz mapping by a transducer}}

In this section, we construct a bi-Lipschitz mapping via the coding space by using a transducer.

\subsection{Dense subsets $E'\subset E$ and $F'\subset F$.}
Let $\mathcal{A}=\{1,\cdots, m\}$ be an alphabet.
Let $\mathcal{A}^{\infty}$ and $\mathcal{A}^{k}$ be the sets of infinite words and words of length $k$ over $\mathcal{A}$, respectively.
Let $\mathcal{A}^{\ast}=\bigcup_{k\geq0} \mathcal{A}^{k}$
be the set of all finite words.
For $\sigma=\sigma_1\cdots \sigma_k\in \mathcal{A}^{\ast} $,
 we use $|\sigma|$ to denote the length of $\sigma$.
If $\rho=\rho_1\cdots \rho_n \in \mathcal{A}^{\ast}$,
we define $\sigma\rho=\sigma_1\cdots \sigma_k \rho_1\cdots \rho_n$.
We shall use $a^\ell$ to denote the word $a\dots a$ consisting of $\ell$ number of $a$,
and understand $a^0$ as the emptyword for convention.
Let  $\omega\wedge \gamma$ be the maximal common prefix of $\omega$ and $\gamma$.

Let $\{\varphi_j\}_{j=1}^m$ be an IFS with attractor $K$.
For $x_1\cdots x_i\in\mathcal{A}^{\ast}$,
denote
$$\varphi_{x_1\cdots x_i}=\varphi_{x_1} \circ \cdots\circ \varphi_{x_i}.$$
Let $\bx=(x_i)_{i=1}^{\infty} \in\mathcal{A}^{\infty}$.
Defined $\pi_{K}:\mathcal{A}^{\infty}\to K$, which we call a \emph{projection},
 by
 \begin{equation}
\big\{\pi_{K}(\bx)\big\}=\bigcap_{i\geq1} \varphi_{x_1\cdots x_i}(K).
\end{equation}
If $\pi_{K}(\bx)=x$,
then we call the sequence $\bx$  a \emph{coding} of $x$.

\begin{figure}
  \includegraphics[height=0.25\textwidth]{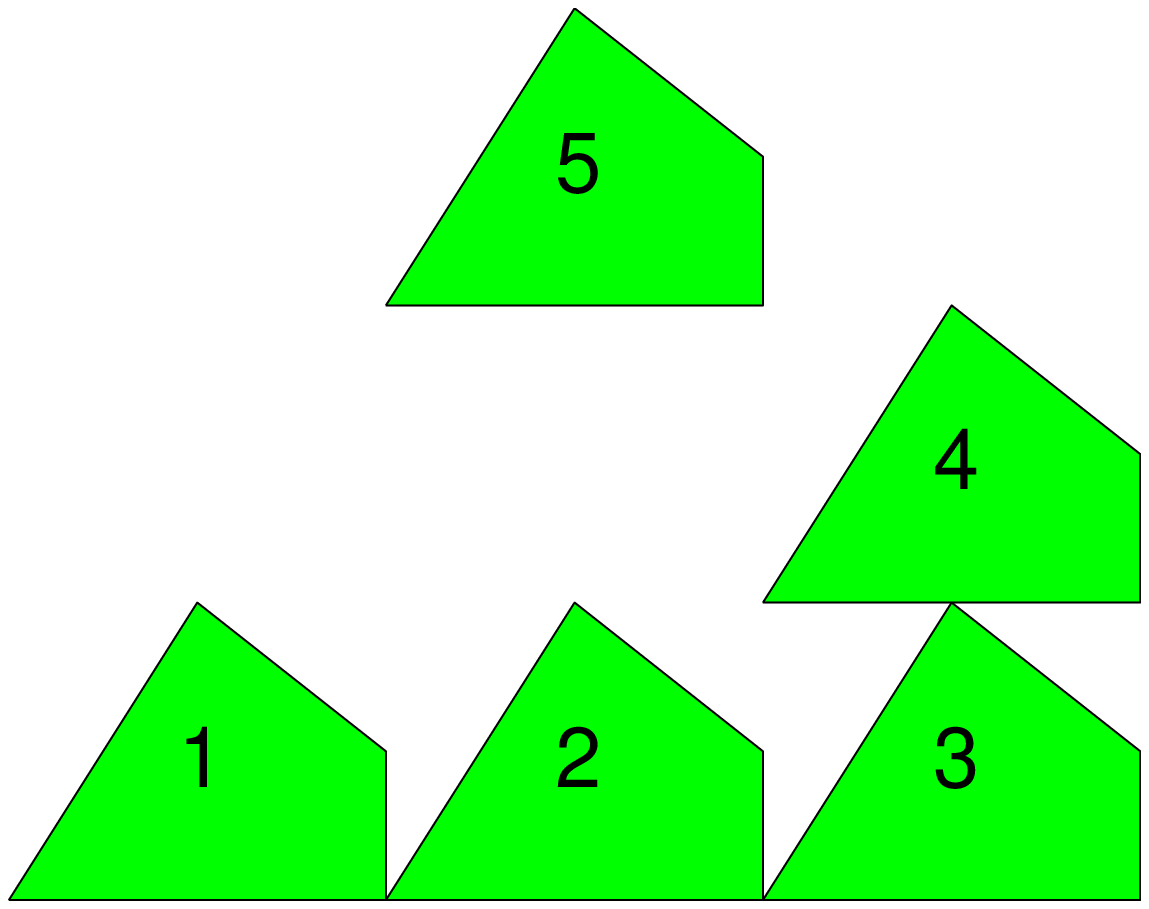}\quad
   \includegraphics[height=0.25\textwidth]{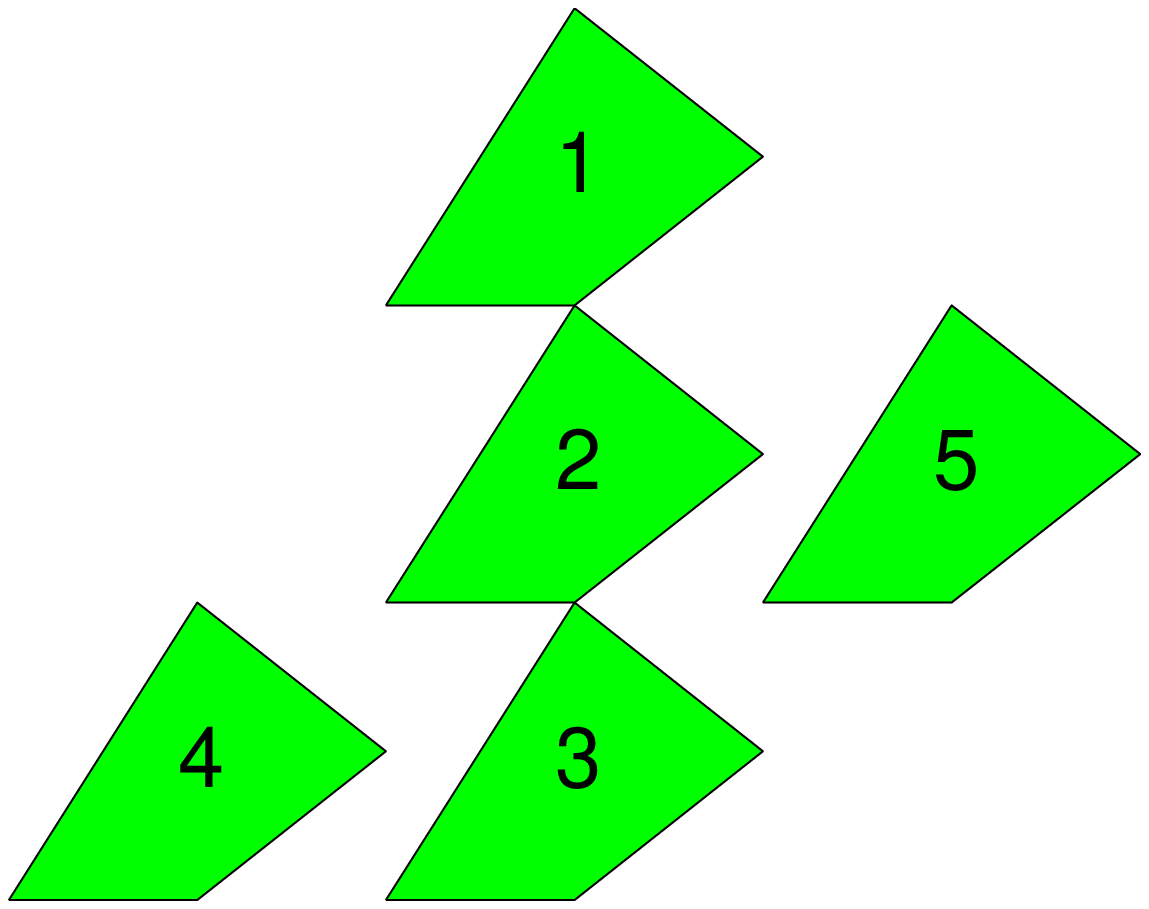}\\
  \caption{Labels of cylinders of $E$ and $F$.}
  \label{fig:labels}
\end{figure}

Recall that $E$ and $F$ are the fractal squares defined in Section 1.
For fractal $E$, we observe that
$$
\pi_{E}(13^{\infty})=\pi_{E}(21^{\infty}),\
\pi_{E}(23^{\infty})=\pi_{E}(31^{\infty}),\
\pi_{E}(35^{\infty})=\pi_{E}(42^{\infty}).
$$
(See Figure \ref{fig:labels}.) Clearly,
a  point $x\in E$ has  double codings  if one of its codings  ending with
$13^{\infty}$, $21^{\infty}$, $23^{\infty}$,
$31^{\infty}$, $35^{\infty}$ and $42^{\infty}$; otherwise, $x$ has a unique coding.
Denote
\begin{equation}
\Omega=\{{\mathbf \omega}4^{\infty};~{\mathbf \omega}\in \{1,\dots,5\}^*\}
\end{equation}
to be the set of sequences ending with $4^\infty$, and set
\begin{equation}
E^{'}=\pi_E(\Omega).
\end{equation}

\begin{remark}{\rm The reason we chose the above subset $E'$ is it is a countable dense subset of $E$, and each point in $E'$ has a unique coding.
 In fact, the other choice of $E'$ satisfying these conditions also works.}
 \end{remark}
 
Similarly, set
\begin{equation}
F'=\pi_F(\Omega),
\end{equation}
then $F'$ is a dense subset of $F$, and each point in $F'$ has a unique coding.
(Notice that in $F$, the basic double coding points are given by  $\pi_F(13^\infty)=\pi_F(21^\infty)$ and  $\pi_F(23^\infty)=\pi_F(31^\infty)$.)

\subsection{Segment decomposition}
Now, we introduce a decomposition of  sequences in $\Omega$, which plays a crucial role in the paper.
\begin{defi}
\rm
Let $\bx=(x_i)_{i=1}^{\infty}  \in \Omega$.~
A factor  $x_j\cdots x_{j+k}$ of $\bx$ with $k\geq 1$  is called a \emph{$E$-special segment},
if one of the following condition holds:

\ $(i)$   $x_j\cdots x_{j+k}=35^k$ and $x_{j+k+1} \neq 5$;

$(ii)$  $x_j\cdots x_{j+k}=42^{k-1}5$.
\end{defi}

The following lemma asserts that the special segments of  a sequence are non-overlapping.
\begin{lemma}\label{lem-non-overlap}
Let $\bx=(x_i)_{i=1}^{\infty}  \in \Omega$.
If  $X_1=x_m\cdots x_{m+k}$ and  $X_2=x_n\cdots x_{n+\ell}$
are two distinct special segments of $\bx$
and $m\leq n$,
then $m+k<n$.
\end{lemma}
\begin{proof}
If both $X_1$ and $X_2$ are of type (i) or are of type (ii), clearly they cannot overlap, for otherwise
one will be a proper prefix of the other, which is impossible. If $X_1$ and $X_2$ are of different type, then
the first letter of $X_2$ does not appear in $X_1$, so they do not overlap.
\end{proof}

A entry $x_j$  of $\bx$ is  called a \emph{$E$-plain segment} ,
if $x_j$ does not  belong to any $E$-special segment.
In what follows,
by a \emph{segment},
we mean  a plain segment or a special segment. Clearly, the collection of all possible $E$-segments is
\begin{equation}
\mathcal{A}_E
    =\big\{35^k;k\geq1\big\}\cup\big\{42^k5;k\geq 0\big\}\cup\big\{1,2,3,4,5\big\}.
\end{equation}

\begin{defi}\rm
Let $\bx=(x_i)_{i=1}^{\infty} \in \Omega$.
By Lemma \ref{lem-non-overlap},
 $\bx$ can be uniquely written as
\begin{equation}\label{2.6}
\bx=\prod_{j=1}^\infty X_j:=X_1X_2\dots
\end{equation}
where $X_j$ are $E$-segments of $\bx$.
We call the right-hand side of  \eqref{2.6} the \emph{$E$-segment decomposition} of $\bx$.
 \end{defi}

For example,
the $E$-segment decomposition of  $235^342^35^2434^{\infty}$ is
$(2)(35^3)(42^35)(5)(4)(3)(4)^{\infty}$.

Similarly, we can define the $F$-segment decomposition of $\bu\in \Omega$.

\begin{defi}\rm For $\bu=(u_i)_{i=1}^\infty$,
a factor $u_j\dots u_{j+k}$ with $k\geq 1$ is called an \emph{$F$-special segment} if
one of the following condition is satisfied:

\ \ $(i)$ $u_j\dots j_{j+k}=42^{k-2}5$ and $j_{j+k+1}\neq 5$;

\ $(ii)$ $u_j\dots j_{j+k}=42^{k-2}55$;

$(iii)$ $u_j\dots j_{j+k}=35$.
\end{defi}

A entry $u_j$ is called an \emph{$F$-plain segment} if it does not belong to any $F$-special segment.
Similar to Lemma \ref{lem-non-overlap}, one can show that any $\bu\in \Omega$ has an $F$-segment decomposition
and it is unique.

\subsection{A mapping $g$ between symbolic spaces}
In the following,
we  construct a map $g$ from $\Omega$ to itself. First, we defined $g_{0}: \mathcal{A}_E \rightarrow \{1,\cdots,5\}^{\ast}$ by
\begin{equation*}
g_0:\left  \{
   \begin{array}{rl}
    35^{k} &\mapsto 42^{k-1}5,~ k\geq 1;\\
    42^{k}5&\mapsto 42^{k-1}55,  ~ k\geq 1;\\
     45&\mapsto 35;  \\
     a&\mapsto a,  ~{\rm    if}~a\in\{1,2,3,4,5\}.
   \end{array}
   \right .
\end{equation*}
Denote $\mathcal{A}_F=g_0\big(\mathcal{A}_E\big)$, then
\begin{equation}\label{AF}
  \mathcal{A}_F=\{42^{k}5;k\geq 0\}\cup\{42^{k}55;k\geq 0\}\cup\{1,2,3,4,5,35\}.
\end{equation}
Clearly,
 $g_0:\mathcal{A}_E\to\mathcal{A}_F$ is a bijection.

 We define $g: \Omega\rightarrow \{1,\cdots,5\}^{\infty}$  by
\begin{equation*}
g(\bx)=\prod_{j=1}^\infty g_0(X_j)
\end{equation*}
where $(X_j)_{j=1}^\infty$ is the $E$-segment decomposition of ${\rm\textbf{x}}$.
Clearly, $g(\Omega)\subset \Omega$.

\begin{lemma}\label{new-1}
 If $(X_j)_{j=1}^\infty$ is the $E$-segment decomposition of $\bx$, then
 the $F$-segment decomposition of $g(\bx)$ is $\prod_{j=1}^\infty g_0(X_j)$.
\end{lemma}

\begin{proof} Denote $\bu=g(\bx)$. We need only show that $g_0(X_j)$ is an $F$-segment of $\bu$ for all $j$.

First, we show that $g$ maps an $E$-special segment to an $F$-special segment.
If $X_j=42^k5$ with $k\geq 1$, then $g(X_j)=42^{k-1}55$ is a $F$-special segment by definition.
The same holds if $X_j=45$.  If $X_j=35^k$, then the letter next to $X_j$, which we denote by $x_m$, cannot be  $5$.
So $g(X_j)=42^{k-1}5$, and the letter next to $g(X_j)$, as the initial letter of $g(X_{j+1})=g(x_m\dots)$, cannot by $5$.
It follows that $g(X_j)$ is an $F$-special segment.

Now we consider the case the $X_j$ is a plain segment. Suppose $g(X_j)$ belongs to an $F$-special segment
$u_p\dots u_{p+k}$. By the conclusion of the preceding paragraph, each $u_i, p\leq i\leq p+k$, is an image of
an $E$-plain segment under $g_0$. Hence  $x_p\dots x_{p+k}=u_p\dots u_{p+k}$ is a factor consisting of $E$-plain segment,
which is impossible since it contains a factor $35$ or $42^{k-1}5$.
\end{proof}

\begin{thm} $g(\Omega)=\Omega$ and $g:\Omega\to \Omega$ is a bijection.
\end{thm}

\begin{proof} First, $g:\Omega\to\Omega$ is an injection. Otherwise,  by Lemma \ref{new-1},
$\bu=g(\bx)=g(\by)$ will have two different $F$-segment decompositions.

Similar to Lemma \ref{new-1}, we can show that if $(U_j)_{j=1}^\infty$ is the $F$-segment decomposition of $\bu$,
then $(g_0^{-1}(U_j))_{j=1}^\infty$ is the $E$-segment decomposition of $\bx=\prod_{j=1}^\infty g_0^{-1}(U_j)$. Consequently,
$\bu=g(\bx)$, which proves $g$ is surjective.
\end{proof}

\subsection{A transducer}
The map $g$ can be realized by the transducer indicated in Figure \ref{Transducer}.
The state set is
$
  \mathcal{M}=\{A,III,III^{'},IV,IV^{'}\}
$
where the  initial state is $A$.
The edge is labeled by $a/w$, where $a$ is the input letter and $w$ is the output word.
From any state, there are five edges going out with input letters $1,2,\dots, 5$ respectively.
Inputting a sequence $(a_n)_{n=1}^\infty$, then a unique path in the transducer in determined,
which we denoted by $(a_n/\omega_n)_{n=1}^\infty$,  and the output sequence is
$\omega_1\omega_2\dots \omega_n\dots$.

\begin{lemma}\label{lem-length}
Let $\bx=(x_i)_{i=1}^{\infty}\in\Omega$ and denote $g(\bx)$ by $\bu=(u_i)_{i=1}^{\infty}$.
Then $u_1\cdots u_n(n\geq1)$ is determined by $x_1\cdots x_{n+1}$.
Moreover,
if $x_{n+1}=1$,
then $u_{n+1}=1$ is also determined.
\end{lemma}

\begin{proof}Observe that the labels $a/w$ satisfies $|a|=|w|=1$, except  the labels of the edges leaving
the state $A$ (with $|a|=1$ and $|w|=0$) or entering the state $A$
(with $|a|=1$ and $|w|=2$) from other states. Moreover, if $a=1$, then we always arrive at the state $A$.
The lemma follows.
\end{proof}

\begin{figure}[h]
  \centering
  \includegraphics[width=.750 \textwidth]{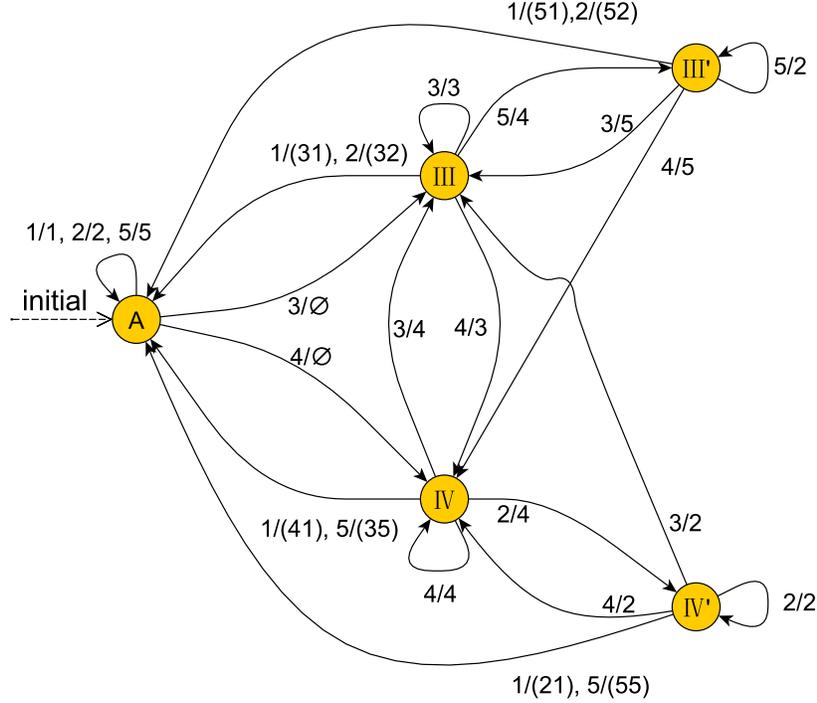}
  \caption{A transducer which realizes the map $g$}
\label{Transducer}
\end{figure}

\begin{remark}\label{g-finite} {\rm Let $J$ be a word ending with $1$, then $J$ determines a path ending at the state $A$,
and hence the transducer gives us an output word of length $|J|$, we shall denote this word by $g(J)$ and we will need this map in Section 4.
}
\end{remark}

\subsection{Construction of a bi-Lipschitz mapping}
We  define $f: E^{'}\to F'$ by
 \begin{equation}
 f(x)=\pi_F\circ g\circ \pi_E^{-1}(x).
 \end{equation}
We shall prove the follow theorem in section 4.

\begin{thm}\label{thm-dense-subset}
 $f: E^{'}\to F^{'}$ is a bi-lipschitz.
\end{thm}

Consequently, the extension $f:E\to F$ is also a bi-lipschitz, which proves Theorem \ref{thm-main}.

\begin{figure}[h]
\centering
\subfigure[\text{The  cylinders $\lfloor 35 \rfloor \cup \lfloor 4\rfloor$ of $E$. }]{
\includegraphics[width=10cm]{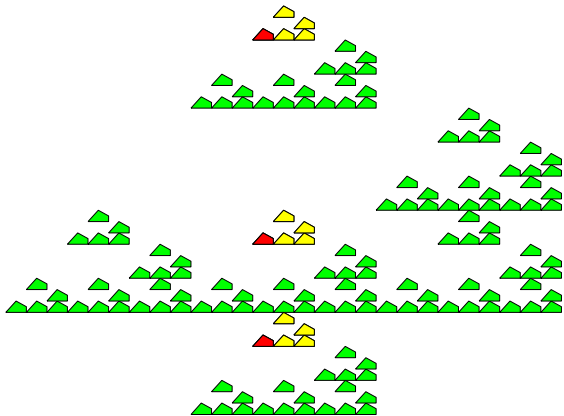}
 }
 \subfigure[The cylinders $\lfloor 3 \rfloor \cup \lfloor 4\rfloor $ of $F$.]{
\includegraphics[width=12cm]{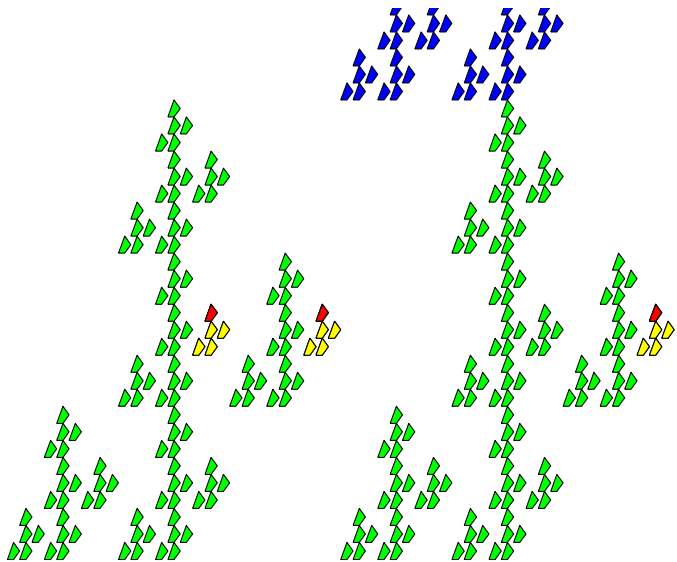}
}
\caption{The cylinders $3551$, $4251$ and $4551$ are indicated by red color.}
\label{fig:EF}
\end{figure}

\begin{remark} \rm Now, we give some intuition of the construction of $f$.
We shall denote $[\bomega]=\{\bomega\bx;~\bx\in \{1,2,\dots,5\}^\infty\}$ and call it a \emph{cylinder}.
If $f$ maps  a cylinder $[\bomega]$  onto $[\bomega']$, we write $[\bomega]_E\mapsto [\bomega']_F$. Denote
$$
X_1=3551, ~X_2=4251, ~X_3=4551.
$$
Then under $f$,
\begin{equation*}
[X_1]_E\mapsto [X_{2}]_F,  ~[X_2]_E\mapsto [X_{3}]_F \text{ and }[X_3]_E\mapsto [X_1]_F.
\end{equation*}
In Figure \ref{fig:EF}, the cylinder $[X_1],[X_2]$ and $[X_3]$ are marked by red color.

In general, denote
\begin{equation*}
X_1=35^k1, X_2=42^{k-1}51, X_3=42^{k-2}5^21,\dots, X_{k+1}=45^k1,
\end{equation*}
then $f$ maps $[X_j]_E$ to $[X_{j+1}]_F$ for $1\leq j\leq k$, and maps $[X_{k+1}]_E$ to $[X_1]_F$.
We also note that
$$dist([X_j]_E, [X_{j+1}]_E)\approx 3^{j-k-1} \text{ for } j=1,\dots, k \text{ and } dist([X_{k+1}]_E, [X_1]_E)\approx 1,$$
and the same relations holds for their image under $f$.
\end{remark}

\section{\textbf{Neighbor graph and separation number}}
Let $K$ be a fractal square generated by the IFS $\Phi=\{\varphi_j\}_{j=1}^m$.
In this section, we introduce an automaton related to $\Phi$
to measure the distant of two points in $K$.

 For $I=i_1\cdots i_k,J=j_1\cdots j_k\in \{1,\cdots, m\}^k$ with $i_1\ne j_1$, we call $\varphi_I^{-1}\circ \varphi_J$ a \emph{neighbor map} if
$\varphi_I(K) \cap \varphi_J(K)\neq \emptyset.$
See for instance, \cite{BanM09}.
We use $\mathcal{N}=\mathcal{N}_{\Phi}$ (or ${\mathcal N}_K$) to denote the collation of all neighbor maps of $\Phi$.
Notice that  a neighbor map $\tau$ must have the form $\tau(x)=x+h$ for some $h\in\mathbb{C}$,
hence we shall denote a neighbor map simply by $h$.

Now we construct a \emph{automaton} relate to $\Phi$ as follows.
The \emph{state set} is
\begin{equation}
  \mathcal{N}^{\ast}=\mathcal{N}\cup\big\{id\big\}\cup\big\{Exit\big\},
\end{equation}
where
$id$ denotes the identity map.
The initial state is $\emph{id}$, and the terminate state is $Exit$.

Next, we define the edge set.
Let $h\in \mathcal{N}\cup\{id\}$.
For any $\displaystyle\binom{i}{j}\in \{1,\cdots,m\}^2$, let
$h'=\varphi_i^{-1}\circ h\circ \varphi_j$.
We define an edge from $h$ to $h'$
if $h^{'}\in\mathcal{N}\cup\{id\}$, and  we define an \emph{edge}  from $h$ to $Exit$ otherwise.
More precisely, we may denote the above edge  by $(h,(i,j))$.
We call the above automaton the $K$-automation.
It is a variation of the neighbor graph, see for instance, \cite{BanM09}.

Take $\bx=(x_i)_{i=1}^{\infty},\by=(y_i)_{i=1}^{\infty}\in \{1,\cdots,m\}^{\infty}$.
We feed the $K$-automaton with $\displaystyle\binom{x_n}{y_n}$ consecutively.
We denote $h_n=\varphi_{x_1\cdots x_n}^{-1}\circ\varphi_{y_1\cdots y_n}$,
and call $(h_n)_{n\geq1}$ the \emph{itinerary} of $\displaystyle\binom{\bx}{\by}$,
and write
\begin{equation*}
  (id)^0\to h_1 \to h_2\to \cdots \to h_n \to \cdots \to Exit.
\end{equation*}
We also say  $\displaystyle\binom{\textbf{x}}{\textbf{y}}$ \emph{arrives} at the state $h_n$ on the $n$-th step.

\begin{defi}
{\rm
For $\bx,\by\in \{1,\cdots,m\}^{\infty}$,
 defined the \emph{separation number}  $\Lambda_{K}(\bx,\by)$ of $\displaystyle\binom{\bx}{\by}$
to be $n$, if $\displaystyle\binom{\bx}{\by}$ arrives at the state $Exit$  on the $(n+1)$-th step.
}
\end{defi}

The following lemma is obvious.

\begin{lemma}
Let $\bx,\by\in\{1,\cdots,m\}^{\infty}$.
If $\varphi_{x_1\cdots x_n}^{-1}\circ\varphi_{y_1\cdots y_n}=h_n\in \mathcal{N}$,
then
\begin{equation}
\varphi_{y_1\cdots y_n}^{-1}\circ\varphi_{x_1\cdots x_n}=-h_n\in\mathcal{N}~{\rm and}~
\Lambda_{K}(\bx,\by)=\Lambda_{K}(\by,\bx).
\end{equation}
\end{lemma}

Now,
we consider the  fractal square $E$ defined by \eqref{1.2}.
We use $\mathcal{N}_E$ and  $\mathcal{N}_E^{\ast}$ to denote the neighbor map set and the state set of $E$-automaton,~respectively.
That is,
 \begin{equation}
  \mathcal{N}_E=\big\{e1,-e1,e2,-e2\big\}~~~{\rm and} ~~~\mathcal{N}_E^{\ast}=\mathcal{N}_E\cup\big\{id\big\}\cup\big\{Exit\big\},
 \end{equation}
where $e1=1$ and $e2=\bi$.
 See Figure \ref{fig:5}.
\begin{figure}[h]
\includegraphics[width=.7 \textwidth]{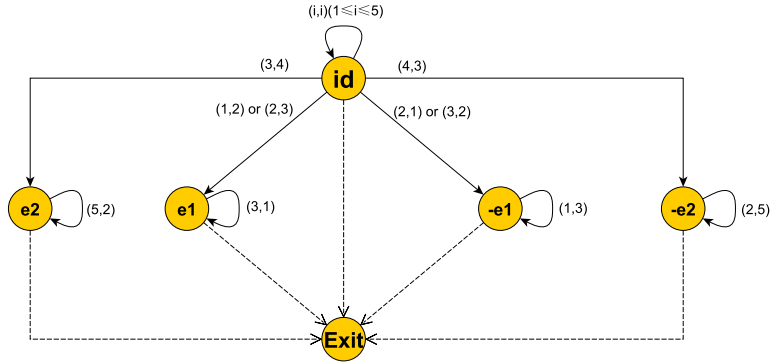}
\caption{$E$-automaton. The edges not labeled on the figure all lead to $Exit$ (along the dotted lines).}
\label{fig:5}
\end{figure}

For $x,y\in E^{'}$,
we shall show that the distance $|x-y|$ is controlled by $\Lambda_E(\bx,\by)$,
where $\bx,\by$ are the coding of $x,y$,~ respectively.

\begin{lemma}\label{lemma 3.2}
For $x,y\in E^{'}$,
let $\bx,\by\in\Omega$ be the coding of $x,y$,~ respectively.
Let $\Lambda_E(\bx,\by)=n$.
Then
\begin{equation*}
\frac{1}{c} 3^{-n}\leq |x-y|\leq c 3^{-n},
\end{equation*}
where $c=6\sqrt{2}$.
\end{lemma}
\begin{proof}
By the definition of separation number,
we have
$S_{x_1\cdots x_{n}}(E)\cap S_{y_1\cdots y_{n}}(E)\neq\emptyset$ and
$S_{x_1\cdots x_{n+1}}(E)\cap S_{y_1\cdots y_{n+1}}(E)=\emptyset$.
Notice that  the convex hull $H$ of $E$  is a quadrilateral (see Figure \ref{fig:labels}(a)), and that ${dist}(S_{x_1\cdots x_{n+1}}(H), S_{y_1\cdots y_{n+1}}(H))>\sqrt{2}/4\cdot 3^{-(n+1)}$.
Hence
$$|x-y|\leq diam\big(S_{x_1\cdots x_{n}}(E)\big)+diam\big(S_{y_1\cdots y_{n}}(E)\big)\leq 2\sqrt{2}\times 3^{-n},$$
and
$$|x-y|\geq dist\big(S_{x_1\cdots x_{n+1}}(E),S_{y_1\cdots y_{n+1}}(E)\big)\geq\frac{\sqrt{2}}{12}\times3^{-n}.$$
 The lemma is proved.
\end{proof}

For the  fractal square $F$, we use $\mathcal{N}_F$ and  $\mathcal{N}_F^{\ast}$ to denote the neighbor map set and the state set of $F$-automaton,~respectively.
That is,
\begin{equation}\label{3.4}
\mathcal{N}_F=\big\{f1,-f1\big\}~~~{\rm and} ~~~\mathcal{N}_F^{\ast}=\mathcal{N}_F\cup\big\{id\big\}\cup\big\{Exit\big\},
\end{equation}
where $f1=-\bi$. See  Figure \ref{fig:6}.

\begin{figure}[h]
\includegraphics[width=.43 \textwidth]{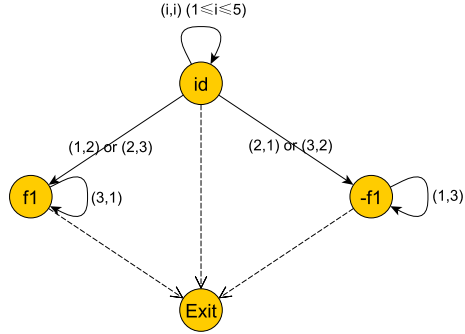}
\caption{$F$-automaton}
\label{fig:6}
\end{figure}

Similar to Lemma \ref{lemma 3.2},
we have the follow lemma.

\begin{lemma}\label{lemma 3.3}
For $u,v\in F^{'}$,
let $\bu,\bv\in \Omega$ be the coding of $u$ and $v$,~respectively.
Let $\Lambda_F(\bu,\bv)=n$.
Then
\begin{equation*}
  \frac{1}{c}3^{-n}\leq|u-v|\leq c3^{-n}
\end{equation*}
where $c=6\sqrt{2}$.
\end{lemma}

\section{\textbf{Proof of Theorem \ref{thm-dense-subset}}}

In this section,
We shall prove Theorem \ref{thm-s-number}, from which  Theorem \ref{thm-dense-subset} will follow immediately.

\begin{thm}\label{thm-s-number}
Let $\bx,\by\in \Omega$, and
denote
$g(\bx)=\bu$ and $g(\by)=\bv$.
Then
\begin{equation}\label{eq-key}
|\Lambda_E(\bx,\by)- \Lambda_F(\bu,\bv)|\leq 3.
\end{equation}
\end{thm}

In the section,
if no otherwise specified,
we always denote
$\bu=g(\bx)~{\rm and}~ \bv=g(\by)$,
and let $(X_j)_{j=1}^{\infty}$, $(Y_j)_{j=1}^{\infty}$
be the $E$-segment decompositions of $\bx, \by$ respectively; let
$(U_j)_{j=1}^{\infty}$, $(V_j)_{j=1}^{\infty}$
be the $F$-segment decompositions of $\bu, \bv$ respectively.

We shall use the following simple fact repeatedly: For $\bx=(x_i)_{i\geq 1}\in \Omega$,
if $x_1=3$, then $X_1=3$ or $35^{k}(k\geq 1)$; if $x_1=4$, then  $X_1=4$, $45$, or $X_1=42^{k}5(k\geq1)$.

\begin{lemma}\label{case-id} For $\bx,\by\in\Omega$,
if $x_1=y_1$ and $X_1\neq Y_1$,
then
\begin{equation}\label{state-id}
 \Lambda_E(\bx,\by)=|\bx\wedge \by| \text{ or } |\bx\wedge \by|+1.
 \end{equation}
 Similarly, if $u_1= v_1$ and $U_1\neq V_1$, then
 \begin{equation}\label{state-id-2}
 \Lambda_F(\bu,\bv)=|\bu\wedge \bv| \text{ or } |\bu\wedge \bv|+1.
 \end{equation}
\end{lemma}

\begin{proof}The  assumptions of the lemma imply that at least one of
 $X_1$ and $Y_1$  is a special segment.
Without loss of generality,
we assume that  $X_1$ is a special segment.
 Let $k$ be the first index such that $x_k\neq y_k$.

 If $k> |X_1|$, then $X_1$ is a prefix of $\by$, which forces
 $X_1=35^{k-2}$ and $Y_1=35^\ell$ with $\ell\geq k-1$. It follows that $y_k=5$, and hence
 the edge $(id,(x_k,y_k))$ leads to $Exit$, so
\begin{equation*}
 \Lambda_E(\bx,\by)=|\bx\wedge \by|.
\end{equation*}

 If $k\leq |X_1|$, then $x_k\in \{2,5\}$ as a non-initial letter of a special segment.
 So $(id,(x_k,y_k))$ either leads to $Exit$ or leads to $\pm e1$.
 If the later case happens, then
 $$(x_k,y_k)\in \{(2,1),(2,3),(1,2),(3,2)\}.$$
 So we must have $x_k=2$. Then, as a non-initial letter of $X_1$,  $x_{k+1}\not \in \{1,3\}$,
 so the $(k+1)$-th step leads to $Exit$, which means
\begin{equation*}
 \Lambda_E(\bx,\by)=k=|\bx\wedge \by|+1
\end{equation*}

The second assertion of the lemma can be proved in the same manner (and it is simpler).
 \end{proof}

Apparently, to prove Theorem \ref{thm-s-number}, we need only prove the theorem for the case that $X_1\neq Y_1$.
We will do this in four lemmas according to $h=Exit, id, \pm e1$ or $\pm e2$, where $h$ is the first state of the
itinerary of $(\bx,\by)$.

\begin{lemma}\label{lem-key-1}
Equation \eqref{eq-key} holds
if $h=Exit$ and $X_1\neq Y_1$.
\end{lemma}

\begin{proof} We note that $h=Exit$ means that $\Lambda_E(\bx,\by)=0$, so $x_1\neq y_1$.
We assert that $u_1\neq v_1$. For otherwise, we must have $(x_1,y_1)=(3,4)$  or $(4,3)$, which implies
that $h=\pm e2$, a contradiction.
Assume that $u_1< v_1$ without loss of generality.

If $\Lambda_F(\bu,\bv)\ge 2$, then
$\displaystyle \binom{u_1u_2}{v_1v_2}=\binom{13}{21} \text{ or }\binom{23}{31}$ since $u_1\neq v_1$.
It follows that $u_1$ and $v_1$ are $F$-plain segments and hence
$$
(x_1,y_1)=(u_1,v_1)=(1,2) \text{ or } (2,3),$$
 which contradicts $\Lambda_E(\bx,\by)=0$. Hence $\Lambda_F(\bu,\bv)=0$ or $1$. The lemma is proved.
\end{proof}

\begin{lemma}\label{lem-key-2}
Equation \eqref{eq-key} holds
if $h=id$ and $X_1\neq Y_1$.
\end{lemma}

\begin{proof} $h=id$ implies that $x_1=y_1$, so
  at least one of
 $X_1$ and $Y_1$  is a special segment.
We may assume that  $X_1$ is a special segment, then
 $X_1=45,35^{k}(k\geq 1) \text{ or }42^{k}5(k\geq1).$

\medskip
\textit{Case 1.}
$X_1=45$.

In this case,
we have  $y_1=x_1=4$, and $y_2\neq 5$ since $X_1\neq Y_1$.
Clearly $\Lambda_E(\bx,\by)=1$.
Now we consider the $F$-automaton. First,
$U_1=g_0(45)=35$.
Since $Y_1=4$ or $Y_1=42^{\ell}5(\ell\geq1)$,
we have that $V_1=g_0(Y_1)=4$ or $42^{\ell-1}55$.
So $(u_1,v_1)=(3,4)$,
which imply that
the itinerary of  $(\bu, \bv)$
is $ (id)^0\to Exit $.
Hence  $\Lambda_F(\bu,\bv)=0$, and \eqref{eq-key} holds in this case.
(By the same argument,   \eqref{eq-key} holds if $Y_1=45$.)

\medskip
\textit{Case 2.}
$X_1=35^{k}(k\geq 1)$.

In this case,
we have $U_1=42^{k-1}5$ and $y_1=3$.
Hence either $Y_1=3$ and $y_2\neq 5$ (if it is a plain segment) or
$Y_1=35^{\ell}(\ell\geq1,\ell\neq k)$ (if it is a special segment).

If $Y_1=3$, clearly $\Lambda_E(\bx,\by)=1$  and $\Lambda_F(\bu,\bv)=0$ since $(u_1,v_1)=(4,3)$.

If $Y_1=35^{\ell}(\ell\geq1,\ell\neq k)$, then $U_1=42^{k-1}5$ and $V_1=42^{\ell-1}5$, so
\begin{equation}\label{eq-ccc}
  |\bx\wedge \by|=\min\{k,\ell\}+1, \quad  |\bu \wedge \bv|=\min\{k,\ell\}\geq 1.
\end{equation}
Hence, \eqref{eq-ccc} together with  Lemma \ref{case-id} imply \eqref{eq-key}.

\medskip

\textit{Case 3.} $X_1=42^{k}5$.

In this case, $U_1=42^{k-1}55$ and  $y_1=4$.
Hence $Y_1=4,45$ or $42^{\ell}5(\ell\geq1,\ell\neq k)$.
The case $Y_1=45$ is proved in  Case 1.

If $Y_1=42^{\ell}5$,  then $U_1=42^{k-1}55$ and $V_1=42^{\ell-1}55$, so
 $$
 |\bx\wedge \by|=\min\{k,\ell\}+1, \quad  |\bu\wedge \bv|=\min\{k,\ell\}\geq 1.
 $$
Hence, by Lemma \ref{case-id}, we have \eqref{eq-key}.

If $Y_1=4$, let $\ell\geq 0$ be the integer such that
\begin{equation*}
y_1\cdots y_{\ell+1}=42^{\ell}(\ell\geq0) \text{ and } y_{\ell+2}\neq 2.
\end{equation*}
Then $y_{\ell+2}\neq 5$ since $Y_1$ is a plain segment, and consequently  $v_{\ell+2}\not \in\{2,5\}$.
Moreover,
$v_1\cdots v_{\ell+1}=42^{\ell}$ since $y_j, 1\leq j\leq \ell+1$ are all plain segments.
Clearly
\begin{equation}\label{eq-aaa}
  |\bx\wedge \by|=\min\{k,\ell\}+1.
\end{equation}
On the other hand, since $(\bu, \bv)=({42^{k-1}55\cdots}, ~{42^{\ell}v_{\ell+2}\cdots})$, we have
\begin{equation}\label{eq-bbb}
|\bu\wedge \bv|=\min\{k-1,\ell\}+1.
\end{equation}
So \eqref{eq-aaa} and \eqref{eq-bbb} together with  Lemma \ref{case-id} imply
 (4.1) .
The proposition is proved.
\end{proof}

\begin{lemma}\label{lem-key-3} Equation \eqref{eq-key} holds
if $h=\pm e1$ and $X_1\neq Y_1$.
\end{lemma}

\begin{proof}  We denote $n=\Lambda_E(\bx,\by)$.  By symmetry, we may assume that $h=e1$.
On the $E$-automaton (Figure \ref{fig:5}),
we observe that the itinerary of $(\bx,\by)$ must be
  $(id)^{0}\to (e1)^{n}\to Exit.$ So
\begin{equation*}
\displaystyle\binom{x_1\cdots x_n}{y_1\cdots y_n}=\binom{13^{n-1}}{21^{n-1}}
~~\text{ or}~~
\displaystyle\binom{23^{n-1}}{31^{n-1}}
~~\text{ and}~~
\displaystyle\binom{x_{n+1}}{y_{n+1}}\neq\binom{3}{1}.
\end{equation*}

If $n\geq 2$,
then
$x_1,\cdots,x_{n-1}$ and  $y_1,\cdots,y_{n}$ are all plain segments,
hence we have
\begin{equation*}
\binom{u_1\cdots u_n}{v_1\cdots v_n}=\binom{x_13^{n-2}3}{y_11^{n-2}1}
~~{\rm or}~~
\binom{x_13^{n-2}4}{y_11^{n-2}1},
\end{equation*}
according to $x_n$ is a plain segment or not.
It follows that
$\Lambda_F(\bu,\bv)\geq n-1$.

If $n=1$,
then
$\Lambda_F(\bu,\bv)\geq0=n-1$.
Thus for all $n\geq 1$,
\begin{equation*}
\Lambda_F(\bu,\bv)\geq n-1
\end{equation*}
always holds.

Suppose $\Lambda_F(\bu,\bv)\geq n+2$, then
\begin{equation*}
\binom{u_1\cdots u_{n+2}}{v_1\cdots v_{n+2}}=\binom{u_13^{n}3}{v_11^{n}1},
\end{equation*}
which implies that
\begin{equation*}
\binom{x_1\cdots x_{n+1}}{y_1\cdots y_{n+1}}=\binom{u_13^{n-1}3}{v_11^{n-1}1},
\end{equation*}
a contradiction. So $\Lambda_F(\bu,\bv)\leq  n+1$ and
 \eqref{eq-key}
holds in this case.
\end{proof}

\begin{lemma}\label{lem-key-4} Equation \eqref{eq-key} holds
if $h=\pm e2$ and $X_1\neq Y_1$.
\end{lemma}

\begin{proof} We denote $n=\Lambda_E(\bx,\by)$.
By symmetry, we may assume that $h=e2$.
On the $E$-automaton, the itinerary of $(\bx,\by)$ must be
  $(id)^{0}\to (e2)^{n}\to Exit.$  So
\begin{equation*}
\displaystyle\binom{x_1\cdots x_n}{y_1\cdots y_n}=\binom{35^{n-1}}{42^{n-1}}
~~\text{and}~~
\displaystyle\binom{x_{n+1}}{y_{n+1}}\neq\binom{5}{2}.
\end{equation*}

If $n\geq 2$, since $u_1\dots u_k$ is determined by $x_1\dots x_{k+1}$ (Lemma \ref{lem-length}), we have
\begin{equation*}
\binom{u_1\cdots u_{n-1}}{v_1\cdots v_{n-1}}=\binom{42^{n-2}}{42^{n-2}},
\end{equation*}
which implies $\Lambda_F(\bu,\bv)\geq n-1$. If $n=1$, this is also true.

Suppose $\Lambda_F(\bu,\bv)> n+2$, then $|\bu\wedge\bv|\geq n+2$ by Lemma \ref{case-id}.
Now at least one of $U_1$ and $V_1$ has length larger than $n+2$ (for otherwise $U_1=V_1$).

If $|U_1|>n+2$, then $U_1=42^k5(k>n)$, so $X_1=35^{k+1}$.
On the other hand, since $v_1\dots v_{n+2}=42^{n+1}$,
we have  $y_1\dots y_{n+1}=42^n$ no matter $Y_1=4$ or $Y_1=42^\ell 5$.
Therefore, $\Lambda_E(\bx,\by)=n+1$,  a contradiction.

If $|V_1|>n+2$, then $V_1=42^{k-1}55(k>n)$, so $Y_1=42^{k}5$. Again,
we have $u_1\dots u_{n+2}=42^{n+1}$,
hence
$x_1\dots x_{n+1}=35^n$ and  $\Lambda_E(\bx,\by)=n+1$.

 Therefore, we have  $\Lambda_F(\bu,\bv)\leq n+2$, and the lemma is proved.
\end{proof}

\medskip

%
%

\noindent \textbf{Proof of Theorem \ref{thm-dense-subset}.}
Pick $x,y\in E^{'}$. Let $\bx, \by$ be the codings of $x,y\in E'$ respectively, and let
$\bu, \bv$ be the codings of $u=f(x), v=f(y)\in F'$ respectively.
By Lemma \ref{lemma 3.2} and Lemma \ref{lemma 3.3},
\begin{equation*}
 \frac{1}{c}3^{-\Lambda_{E}(\bx,\by)}\leq|x-y|\leq c 3^{-\Lambda_{E}(\bx,\by)},
\end{equation*}
and
\begin{equation*}
\frac{1}{c}3^{-\Lambda_{F}(\bu,\bv)}\leq|u-v|\leq c 3^{-\Lambda_{F}(\bu,\bv)},
\end{equation*}
where  $c=6\sqrt{2}$.
Using Theorem \ref{thm-s-number}, we obtain
\begin{equation*}
  \frac{1}{27c^2}|x-y|\leq|u-v|\leq 27c^2|x-y|.
\end{equation*}
Hence $f: E^{'}\rightarrow F^{'}$ is a bi-Lipschitz map.  $\Box$
 \medskip
 
 Extending the map $f:~E'\to F'$ to $E$ by continuity, we still denote it by $f$.

\noindent \textbf{Proof of Theorem \ref{thm-main}.} 
We have seen that $f:~E\to F$ is bi-Lipschitz.  Now show that $f:~E\to F$ is measure-preserving.
To this end, we need only show that 
\begin{equation}\label{mea}
\frac{{\mathcal H}^s(E_I)}{{\mathcal H}^s(f(E_I))}=\frac{{\mathcal H}^s(E)}{{\mathcal H}^s(F)},  \text{ for any word } I\in \{1,2,3,4,5\}^*
\end{equation}
 where $s=\dim_H E=\log 5/\log 3$. Let
 $$
 {\mathcal J}=\{J'41;~ J'\in \{1,2,3,,4,5\}^* \text{ and `$41$' does not appear in }J'\}.
 $$
Then for any $J\in {\mathcal J}$, the cylinder $E_J$ maps onto a cylinder of $F$
of the same order, precisely, $f(E_J)=F_{g(J)}$ where $g$ is the map defined by the transducer in Section 2 (see Remark \ref{g-finite});
and hence \eqref{mea} holds.

Notice that $E_J$, $J\in {\mathcal J}$ are disjoint in Hausdorff measure, and
$$
{\mathcal H}^s(E)=\sum_{J\in {\mathcal J}} {\mathcal H}^s(E_J).
$$
It follows that for any $I\in \{1,2,3,4,5\}^*$,
\begin{equation}\label{eq-disjoint}
{\mathcal H}^s(E_I)=\sum_{J\in {\mathcal J}} {\mathcal H}^s(E_{IJ}).
\end{equation}
Moveover,
\begin{equation}\label{eq-g}
{\mathcal H}^s(f(E_{IJ}))={\mathcal H}^s(F_{g(IJ)})=c {\mathcal H}^s(E_{IJ}),
\end{equation}
where $c={\mathcal H}^s(F)/{\mathcal H}^s(E)$. Therefore, 
$$
\begin{array}{rl}
{\mathcal H}^s(f(E_I))& =\sum_{J\in {\mathcal J}} {\mathcal H}^s(f(E_{IJ})) \quad (\text{ by } \eqref{eq-disjoint})\\
&=c\sum_{J\in {\mathcal J}} {\mathcal H}^s(E_{IJ}) \quad (\text{ by } \eqref{eq-g})\\
&=c{\mathcal H}^s(E_I), \quad (\text{ by } \eqref{eq-disjoint})
\end{array}
$$
which proves \eqref{mea}, and finishes the proof of the theorem.
$\Box$

\end{document}